\documentclass[11pt]{amsart}

\usepackage{hyperref}
\usepackage{amsthm,amsmath,amssymb,amscd,amsfonts}
\usepackage{mathrsfs}
\usepackage[style=alphabetic,backend=biber]{biblatex}

\bibliography{autnn}

\theoremstyle{plain}
\newtheorem{theorem}{Theorem}[section]
\newtheorem{proposition}[theorem]{Proposition}
\newtheorem{corollary}[theorem]{Corollary}
\newtheorem{lemma}[theorem]{Lemma}
\newtheorem{definition}[theorem]{Definition}
\newtheorem*{tha}{Theorem A}
\newtheorem*{thb}{Theorem B}
\newtheorem*{thc}{Theorem C}
\newtheorem*{thd}{Theorem D}
\newtheorem*{thm}{General Theorem}

\theoremstyle{definition}
\newtheorem{remark}[theorem]{Remark}

\newcommand{\sq}[1]{\ifx#1([\else\ifx#1)]%
  \else\message{invalid use of "sq"}\fi\fi}
\newcommand{\exc}{{\operatorname{exc}}}

\DeclareMathOperator{\ord}{ord}
\DeclareMathOperator{\Div}{Div}
\DeclareMathOperator{\Supp}{Supp}

\DeclareMathOperator{\Nev}{Nev}
\DeclareMathOperator{\Nevbir}{Nev_{\operatorname{bir}}}
\DeclareMathOperator{\NevEF}{Nev_{\operatorname{EF}}}
\DeclareMathOperator{\Spec}{Spec}
\DeclareMathSymbol{\idot}{\mathbin}{operators}{`\.}

\allowdisplaybreaks
\begin{document}
\title{A birational Nevanlinna constant and its consequences}

\author{Min Ru and Paul Vojta}

\address{Department of Mathematics\\University of Houston\\4800 Calhoun Road, Houston, TX 77204\\USA;
 Department of Mathematics, University of California,
 970 Evans Hall \#3840, Berkeley, CA  94720-3840, USA}

\email{minru@math.uh.edu, vojta@math.berkeley.edu}

\subjclass[2010]{14G25, 32H30, 14C20, 11J97, 11G35}
\keywords{Nevanlinna constant; Schmidt's subspace theorem; integral points}

\begin{abstract}
The purpose of this paper is to modify the notion of
the Nevanlinna constant $\Nev(D)$, recently introduced by the first author,
for an effective Cartier divisor on a projective variety $X$.
The modified notion is called the \emph{birational Nevanlinna constant}
and is denoted by $\Nevbir(D)$.
By computing $\Nevbir(D)$ using the filtration constructed by
Autissier in 2011, we establish a general result (see the
General Theorem in the Introduction), in both the arithmetic and
complex cases, which extends to general divisors the 2008 results of
Evertse and Ferretti and the 2009 results of the first author.
The notion $\Nevbir(D)$
is originally defined in terms of Weil functions for use in applications,
and it is proved later in this paper that it can be defined in terms of local
effectivity of Cartier divisors after taking a proper birational lifting.
In the last two sections, we use the notion $\Nevbir(D)$ to recover the proof
of an example of Faltings from his 2002 \emph{Baker's Garden} article.
\end{abstract}

\maketitle

\section{Introduction}\label{intro}
We consider the following questions: {\it For a given effective Cartier divisor $D$ on a given projective variety $X$, find the conditions (for $D$ and $X$) such that every holomorphic mapping   $f\colon\mathbb C\to X\backslash D$ must be degenerate (i.e. its image is contained some proper subvariety of $X$);  If both $D$ and $X$ are defined over a number field $k$, then one also asks when every set of
integral points of $X\backslash D$ must be degenerate.} In answering the above questions, the first author
introduced (\cite{R4} and \cite{ru}) the notion of the \emph{Nevanlinna constant}, denoted by $\Nev(D)$, and proved that {\it if $\Nev(D)<1$, then every holomorphic mapping   $f\colon\mathbb C\to X\backslash D$ must be degenerate, and every set of integral points of $X\backslash D$ must also be degenerate if  both $D$ and $X$ are defined over a number field $k$}. Moreover, the quantitative versions of the above
results, in the spirit of Nevanlinna--Roth--Cartan--Schmidt, are also obtained.

We recall his result in detail here.
For notations see Sections \ref{prelim} and \ref{sect_weil}.

\begin{definition}[see \cite{R4} and \cite{ru}] \label{def_nevintro}
Let $X$ be a normal projective variety, and let $D$ be an effective
Cartier divisor on $X$.  The \textbf{Nevanlinna constant} of $D$,
denoted $\Nev(D)$, is given by
\begin{equation}
  \Nev(D) = \inf_{N,V,\mu} \frac{\dim V}{\mu}\;,
\end{equation}
where the infimum is taken over all triples $(N,V,\mu)$ such that
$N$ is a positive integer, $V$ is a linear subspace of $H^0(X,\mathscr L^N)$
with $\dim V\ge2$, and $\mu$ is a positive real number such that,
for all $P\in \Supp D$, there exists a basis $\mathcal B$ of $V$ with
\begin{equation} \label{d}
  \sum_{s\in \mathcal B} \ord_E (s) \ge \mu \ord_E (ND)
\end{equation}
for all irreducible components $E$ of $D$ passing through $P$.
If $\dim H^0(X,\mathscr O(ND))\leq 1$ for all positive integers $N$,
we define $\Nev(D)=+\infty$.  For a general complete variety $X$,
$\Nev(D)$ is defined by pulling back to the normalization of $X$.
\end{definition}

\begin{tha} [see \cite{R4}] \label{masterc}
Let $X$ be a complex projective variety, and let $D$ be an effective
Cartier divisor on $X$.  Then, for every $\epsilon>0$,
$$ m_f(r, D) \le_\exc \left(\Nev(D)+\epsilon\right) T_{f, D}(r) $$
holds for any holomorphic mapping $f\colon \mathbb C\rightarrow X$ with
Zariski-dense image.  Here the notation $\le_\exc$ means that the inequality
holds for all $r\in(0,\infty)$ outside of a set of finite Lebesgue measure.
\end{tha}

\begin{thb}[see \cite{ru}] \label{mastera}
Let $k$ be a number field, and let $S$ be a finite set of places of $k$
containing all of the archimedean places.  Let $X$ be a projective variety over $k$,
and let $D$ be an effective Cartier divisor on $X$.
Then, for every $\epsilon>0$, the inequality
$$ m_S(x, D) \leq \left(\Nev(D)+\epsilon\right) h_D(x) $$
holds for all $x\in X(k)$ outside a proper Zariski closed subset of $X$.
\end{thb}

As was shown in \cite{R4} and \cite{ru},
by computing $\Nev(D)$,  the above  results recover the results of  Evertse--Ferretti  \cite{ef_festschrift} and
of  \cite{ru_annals}, as well as derive   new results for divisors which are not necessarily linear equivalent on $X$. More
importantly, it led to a unified proof (for the known results)  by simply computing $\Nev(D)$.

In attempting to use the filtration constructed by Autissier in \cite{Aut2}
to derive a more general result (see the General Theorem below),
as well as in deriving a proof of an example of Faltings \cite{faltings}
from the view of $\Nev(D)$,  the authors realized that
the notion of $\Nev(D)$ is not general enough to serve our purpose.
More specifically, as we shall see,
the (pointwise) maximum of two or more Weil functions occurs in the proofs,
and this is not in general a Weil function. All of these facts motivate
the following modified definition.

Let $\mathscr L$ be a line sheaf on $X$ and let $\mathcal B$ be
a finite set of global sections of $\mathscr L$.
Then we let $(\mathcal B)$ denote the sum of the divisors $(s)$
for all $s\in\mathcal B$:
\begin{equation}\label{def_parens_B}
  (\mathcal B) = \sum_{s\in\mathcal B} (s)\;.
\end{equation}

\begin{definition}\label{bidef1}
Let $X$ be a normal complete variety, let $D$ be an effective Cartier divisor
on $X$, and let $\mathscr L$ be a line sheaf on $X$.  Then
$$\Nevbir(\mathscr L,D) = \inf_{N,V,\mu} \frac{\dim V}{\mu}\;,$$
where the infimum passes over all triples $(N,V,\mu)$ such that
$N\in\mathbb Z_{>0}$, $V$ is a linear subspace of
$H^0(X,\mathscr L^N)$
with $\dim V>1$, and $\mu\in\mathbb Q_{>0}$, with the following property.
There are finitely many bases $\mathcal B_1,\dots,\mathcal B_\ell$
of $V$; Weil functions $\lambda_{\mathcal B_1},\dots,\lambda_{\mathcal B_\ell}$
for the divisors $(\mathcal B_1),\dots,(\mathcal B_\ell)$, respectively;
a Weil function $\lambda_D$ for $D$; and an $M_k$-constant $c$ such that
\begin{equation}
  \max_{1\le i\le\ell}\lambda_{\mathcal B_i} \ge \mu N\lambda_D - c
\end{equation}
(as functions $X(M_k)\to\mathbb R\cup\{+\infty\}$).  (Here we use the same
convention as in Definition \ref{def_nevintro} when there are no triples
$(N,V,\mu)$ that satisfy the condition.)

If $L$ is a Cartier divisor or Cartier divisor class on $X$, then
we define $\Nevbir(L,D)=\Nevbir(\mathscr O(L),D)$.  We also define
$\Nevbir(D)=\Nevbir(D,D)$.
\end{definition}

\begin{remark}
Since $\Nevbir(\mathscr L,D)$ is a birational invariant
(Proposition \ref{prop_b_nev}(a)), and
since blowing up turns a proper closed subscheme into an effective Cartier
divisor, the restriction that $D$ be an effective Cartier divisor is
unnecessary.  Indeed, we can use the definitions of Silverman \cite{silverman}
or Yamanoi \cite[Def.~2.2.1]{yamanoi} to define Weil functions for proper
closed subschemes; one may then allow $D$ to be a proper closed subscheme
(or a finite linear combination of proper closed subschemes
with coefficients in $\mathbb N$, or even in $\mathbb R^+$).
\end{remark}

With the notation  $\Nevbir(\mathscr L, D)$, we modify Theorems A and B
as follows:

\begin{theorem}\label{b_thmd}
Let $X$ be a complex projective variety, let $D$ be an effective Cartier
divisor and $\mathscr L$ be a line sheaf on $X$ with $\dim H^0(X, \mathscr L^N)\ge 1$ for some $N>0$. Let $f\colon\mathbb C\to X$ be a holomorphic mapping with
Zariski-dense image.  Then, for every $\epsilon>0$,
\begin{equation}\label{b_thmc_eq1}
  m_f(r, D) \le_\exc \left(\Nevbir(\mathscr L, D)+\epsilon\right) T_{f, \mathscr L}(r).
\end{equation}
\end{theorem}

\begin{theorem}\label{b_thmc}
Let $k$ be a number field, and let $S$ be a finite set of places of $k$
containing all archimedean places.
Let $X$ be a projective variety over $k$, and
 let $D$ be an effective Cartier
divisor on $X$.
Let $\mathscr L$ be a line sheaf on $X$ with $\dim H^0(X,\mathscr L^N)\ge 1$
for some $N>0$.
Then, for every $\epsilon>0$,
there is a proper Zariski-closed subset $Z$ of $X$ such that the inequality
\begin{equation}
  m_S(x, D) \le \left(\Nevbir(\mathscr L, D)+\epsilon\right) h_{\mathscr L}(x)
\end{equation}
holds for all $x\in X(k)$ outside of $Z$.
\end{theorem}

\begin{corollary}
Let $X$ be a projective variety over a  number field $k$, and let $D$ be an ample Cartier divisor
on $X$. If $\Nevbir(D)<1$ then there is a proper Zariski-closed subset $Z$
of $X$ such that any set of $D$-integral points on $X$ has only finitely many
points outside of $Z$.
\end{corollary}

We note that while the above definition of $\Nevbir(\mathscr L,D)$ is
convenient for use in applications, it
involves Weil functions in its definition.  As we shall see later,
one  can actually get rid of the Weil functions in
the  definition of $\Nevbir(\mathscr L,D)$  by taking a proper birational lifting.  So we
 propose the  second (equivalent) definition $\Nevbir(\mathscr L, D)$.

\begin{definition}\label{bidef2}
Let $X$ be a complete variety, let $D$ be an effective Cartier divisor
on $X$, and let $\mathscr L$ be a line sheaf on $X$.  If $X$ is normal, then
we define
$$\Nevbir(\mathscr L,D) = \inf_{N,V,\mu} \frac{\dim V}{\mu}\;,$$
where the infimum passes over all triples $(N,V,\mu)$ such that
$N\in\mathbb Z_{>0}$, $V$ is a linear subspace of
$H^0(X,\mathscr L^N)$
with $\dim V>1$, and $\mu\in\mathbb Q_{>0}$, with the following property.
There are a variety $Y$ and a proper birational morphism $\phi\colon Y\to X$
such that the following condition holds.
For all $Q\in Y$ there is a basis $\mathcal B$ of $V$ such that
\begin{equation}
  \phi^{*}(\mathcal B) \ge \mu N\phi^{*}D
\end{equation}
in a Zariski-open neighborhood $U$ of $Q$,
relative to the cone of effective $\mathbb Q$-divisors on $U$.
If there are no such triples $(N,V,\mu)$, then $\Nevbir(\mathscr L,D)$
is defined
to be $+\infty$. For a general complete variety $X$, $\Nevbir(\mathscr L,D)$
is defined by pulling back to the normalization of $X$.

\end{definition}

Note that a \emph{birational morphism} from $X$ to $Y$ is a morphism $X\to Y$
that has an inverse as a rational map; in other words, it is a birational map
$X\dashrightarrow Y$ that is regular everywhere on $X$.

One of the goals of this paper is to prove that these two definitions
(Definitions \ref{bidef1} and \ref{bidef2}) are equivalent.

Another goal of this paper is to use Theorems \ref{b_thmc} and \ref{b_thmd},
together with the filtration constructed by
Autissier in \cite{Aut2}, to prove the following two General Theorems
(in the arithmetic and analytic cases).

Throughout this section, we use $h^0(\mathscr L)$ to denote $\dim H^0(X, \mathscr L)$ for a line sheaf $\mathscr L$ on $X$,
$h^0(D)$ to denote $\dim H^0(X, \mathscr O(D))$
for an effective divisor $D$ on $X$, and $H^0(X, \mathscr L(-D))$ to denote $H^0(X, \mathscr L\otimes \mathscr O(-D))$.

\begin{definition}\label{def_aut_lambda}
Let  $\mathscr L$ be a line sheaf and  $D$  be a nonzero effective Cartier divisor on a complete
variety $X$.  We define
\begin{equation}
  \gamma(\mathscr L, D) = \inf_N {Nh^0(\mathscr L^N)\over  \sum_{m\ge 1}h^0(\mathscr L^N(-mD))},
\end{equation}
where $N$ passes over all positive integers such that $h^0(\mathscr L^N(-D))\ne0$.
(Note that $|\mathscr L^N|$ does not have to be base point free). If $h^0(\mathscr L^N(-D))=0$ for all $N$, then we set $  \gamma(\mathscr L, D)=+\infty$.
\end{definition}

\begin{thm} [Arithmetic Part] \label{Ga}
Let $X$ be a projective variety over a number field $k$, and
let $D_1, \dots, D_q$ be effective Cartier divisors intersecting properly
on $X$.  Let $D=D_1+\dots+D_q$.  Let  $\mathscr L$ be a line sheaf on $X$
with $h^0(\mathscr L^N)\ge 1$ for $N$ big enough.
 Let $S\subset M_k$ be a finite set of places. Then,
 for every $\epsilon>0$,
$$ m_S(x, D) \leq \left(\max_{1\leq j \leq q} \gamma(\mathscr L, D_j)+\epsilon\right) h_{\mathscr L}(x),$$
holds for all $k$-rational points  outside a proper Zariski closed
subset  of $X$.
\end{thm}

\begin{thm}[Analytic Part] \label{Gb}
Let $X$ be a complex projective variety and let $D_1, \dots, D_q$ be
effective Cartier divisors intersecting properly on $X$.  Let $D=D_1+\dots+D_q$.
Let  $\mathscr L$ be a line sheaf on $X$ with $h^0(\mathscr L^N)\ge 1$ for $N$ big enough.
Let $f: \mathbb C \rightarrow X$ be an algebraically non-degenerate
holomorphic map.
 Then, for every $\epsilon > 0$,
$$m_f(r, D) \le_\exc  \left(\max_{1\leq j \leq q} \gamma(\mathscr L, D_j)+\epsilon\right)T_{f, \mathscr L}(r). $$
\end{thm}

Write $\gamma(D_j): =\gamma(\mathscr O(D), D_j)$ with $D:=D_1+\cdots+D_q$.
To  compute $\gamma(D_j)$, we consider $D_1, \dots, D_q$ which are  effective Cartier divisors on $X$ in general position,
such that each $D_j$ is linearly equivalent to  a fixed ample divisor $A$.
By the Riemann-Roch theorem, with $n=\dim X$,
$$h^0(ND) = h^0(qNA)={(qN)^n A^n\over n!} + o(N^n)$$
and
$$h^0(ND-mD_j) = h^0((qN-m)A)={(qN-m)^n A^n\over n!} + o(N^n).$$
Thus $$\sum_{m\ge 1} h^0(ND-mD_j) ={A^n\over n!}\sum_{l= 0}^{qN-1} l^n+o(N^{n+1})
={A^n(qN-1)^{n+1}\over (n+1)!} + o(N^{n+1}).$$
Hence
$$\gamma(D_j)=\lim_{N\rightarrow \infty} {N{(qN)^n A^n\over n!}+ o(N^{n+1})\over {A^n(qN-1)^{n+1}\over (n+1)!} +o(N^{n+1})}={n+1\over q}.$$
Therefore, the General Theorems imply the following results of Evertse--Ferretti in the case when  $X$ is Cohen--Macaulay (for example $X$ is smooth)
as well as the result of Ru.

\begin{thc} [Evertse--Ferretti \cite{ef_festschrift}]
Let $X$ be a projective variety over a number field $k$,  and let
$D_1, \dots, D_q$ be effective divisors on $X$ in general position.
Let $S\subset M_k$ be a finite set of places.
Assume that there exist an ample divisor $A$ on $X$ and positive integers
$d_i$ such that $D_i$ is linearly equivalent to $d_i A$ for $i=1, \dots, q$.
Then, for every  $\epsilon >0$,
$$\sum_{i=1}^q {1\over d_i} m_S(x, D_i) \leq (\dim X+1+\epsilon) h_A(x)$$
holds for all $k$-rational points outside a proper Zariski closed
subset of $X$.
\end{thc}

\begin{thd}[\cite{ru_annals}]
Let $X$ be a smooth complex projective variety and
$D_1, \dots, D_q$ be  effective divisors on $X$, located in general position.
Suppose that there exists an ample divisor $A$ on $X$
and positive integers $d_i$ such that    $D_i$ is linearly equivalent to $d_i A$ on $X$ for
 $i=1, \dots, q$.
Let $f: \mathbb C \rightarrow X$ be an algebraically non-degenerate
 holomorphic map.
 Then, for every $\epsilon > 0$,
$$ \sum_{i=1}^q {1\over d_i} m_f(r, D_i) \le_\exc (\dim X+1+\epsilon)T_{f,A}(r).$$
\end{thd}

More computations of $\gamma(D_i)$ and consequences of the General Theorems
will be given later.

It is important to note that our General Theorem can be proved without using
 the notion of $\Nevbir(\mathscr L, D)$ (and thus not using  Theorem \ref{b_thmc} and Theorem \ref{b_thmd})
by, instead, applying the Schmidt's subspace theorem and H. Cartan's theorem directly (see the proof below). More importantly
 our proof  greatly simplifies the
original proof of  Evertse--Ferretti  and of Ru, where the Chow and Hilbert weights are involved.

Our paper is organized as follows.
In Section \ref{prelim}, we introduce notations and preliminaries.
In Section \ref{sect_weil}, we review the definition and properties of Weil functions.
In Section \ref{sect_aut}, we give a proof of our General Theorem
as well as discuss its application in the case when $D=D_1+\cdots +D_q$
has equi-degree.
Sections \ref{sect_nevbir} and \ref{models} are devoted to proving that
the two definitions of $\Nevbir(D)$ given in this section are equivalent.
Section \ref{proofs} gives a proof of the main theorems involving $\Nevbir(D)$.
Finally, in Sections \ref{falt_geom} and \ref{falt_finis}, we revisit
the example of Faltings \cite{faltings},
since its proof gave the original motivation for our modified
Nevanlinna constant.

We  will primarily discuss the number field cases of our results here.
The corresponding results in Nevanlinna theory can be proved by similar
methods, as can the diophantine results over function fields of characteristic
zero.

\section{Notation and Preliminaries}\label{prelim}

In this paper, $\mathbb N=\{0,1,2,\dots\}$ and $\mathbb R^+$ is the interval
$\sq(0,\infty)$.

\subsection{Notation and conventions in number theory}
For a number field $k$, recall that $M_k$ denotes the set of places of $k$, and
that $k_\upsilon$ denotes the completion of $k$ at a place $\upsilon\in M_k$.
Norms $\|\cdot\|_\upsilon$ on $k$ are normalized so that
$$\|x\|_\upsilon = |\sigma(x)|^{[k_\upsilon:\mathbb R]} \qquad\text{or}\qquad
  \|p\|_\upsilon = p^{-[k_\upsilon:\mathbb Q_p]}$$
if $\upsilon\in M_k$ is an archimedean place corresponding to an embedding
$\sigma\colon k\hookrightarrow\mathbb C$ or a non-archimedean place lying
over a rational prime $p$, respectively.

Heights are logarithmic and relative to the number field used as a base field,
which is always denoted by $k$.  For example, if $P$ is a point
on $\mathbb P^n_k$ with homogeneous coordiantes $[x_0:\dots:x_n]$ in $k$, then
$$h(P) = h_{\mathscr O(1)}(P)
  = \sum_{\upsilon\in M_k} \log\max\{\|x_0\|_\upsilon,\dots,\|x_n\|_\upsilon\}
  \;.$$

We use the standard notations of Nevanlinna theory and Diophantine approximation
(see, for example, \cite{vojta_cm}, \cite{R4} and \cite{ru}).

\subsection{Notation in algebraic geometry}
A \emph{variety} over a field $k$ is an integral scheme,
separated and of finite type over $\Spec k$.  A morphism of varieties is
a morphism of schemes over $k$.  A \emph{line sheaf} is a locally free sheaf
of rank $1$ (an invertible sheaf).

If $\mathscr L$ is a line sheaf on a variety $X$, then $\mathscr L^n$
denotes the $n^{\text{th}}$ tensor power $\mathscr L^{\otimes n}$,
and if $D$ is a Cartier divisor on $X$, then $\mathscr L(D)$ denotes
$\mathscr L\otimes\mathscr O(D)$.

\begin{definition}\label{def_int_properly}
Let $D_1,\dots,D_q$ be effective Cartier divisors on a variety $X$
of dimension $n$.
\begin{enumerate}
\item[(a).]  We say that $D_1,\dots,D_q$ lie in \textbf{general position}
if for any $I\subseteq \{1,\dots,q\}$, we have
$\dim (\bigcap_{i\in I} \Supp D_i)=n-\#I$ if $\#I\le n$, and
$\bigcap_{i\in I} \Supp D_i=\emptyset$ if $\#I>n$.
\item[(b).]  We say that $D_1,\dots,D_q$ \textbf{intersect properly}
if for any subset $I\subseteq \{1, \dots, q\}$ and
any $x\in \bigcap_{i\in I} \Supp D_i$, the sequence $(\phi_i)_{i\in I}$
is a regular sequence in the local ring ${\mathscr O}_{X,x}$,
where $\phi_i$ are the local defining functions of $D_i$, $1 \leq i \leq q$.
\end{enumerate}
\end{definition}

\begin{remark}\label{properint}
By Matsumura \cite[Thm.~17.4]{mat}, if $D_1,\dots,D_q$ intersect properly, then
they lie in general position.  The converse holds if $X$ is Cohen--Macaulay
(this is true if $X$ is nonsingular).
\end{remark}

\subsection{Theorems of Schmidt and Cartan}
Schmidt's Subspace Theorem and the corresponding theorem of Cartan play a
central role in this paper.  This subsection recalls the specific variants of
these theorems to be used in this paper.

We start with some notation, starting with the number field case.
Let $k$ be a number field, and let $H$ be a hyperplane in $\mathbb P^n_k$.
Let $a_0x_0+\dots+a_nx_n$ be a linear form with $a_0,\dots,a_n\in k$ whose
vanishing determines the hyperplane $H$.  Then, for all $v\in M_k$ and all
rational points $P\in\mathbb P^n_k(k)$ with homogeneous coordinates
$[x_0:\dots:x_n]$, we let
$$\lambda_{H,v}(P) = -\log\frac{\|a_0x_0+\dots+a_nx_n\|_\upsilon}
    {\max\{\|a_0\|_\upsilon,\dots,\|a_n\|_\upsilon\}
      \cdot\max\{\|x_0\|_\upsilon,\dots,\|x_n\|_\upsilon\}}\;.$$
This quantity is independent of the choice of homogeneous coordinates for $P$,
and also does not depend on the linear form $a_0x_0+\dots+a_nx_n$ chosen above.

Then Schmidt's Subspace theorem is as follows (see \cite[Thm.~2.8]{ru},
for example).

\begin{theorem}\label{schmidt}
Let $k$ be a number field, let $S$ be a finite set of places of $k$ containing
all archimedean places, let $n$ be a positive integer, let $H_1,\dots,H_q$
be hyperplanes in $\mathbb P^n_k$, let $\epsilon>0$, and let $c\in\mathbb R$.
Then there is a finite union $Z$ of proper linear subspace of $\mathbb P^n_k$,
depending only on $k$, $S$, $n$, $H_1,\dots,H_q$, $\epsilon$, and $c$,
such that the inequality
$$\sum_{\upsilon\in S}\max_J\sum_{j\in J}\lambda_{H_j,\upsilon}(x)
  \le (n+1+\epsilon)h(x) + c$$
holds for all $x\in(\mathbb P^n_k\setminus Z)(k)$.  Here the set $J$ ranges
over all subsets of $\{1,\dots,q\}$ such that the hyperplanes $(H_j)_{j\in J}$
lie in general position.
\end{theorem}

The corresponding theorem for approximation to hyperplanes by
holomorphic curves is due to Cartan \cite{cartan}.

Again, we first need a definition.  Let $H$ be a hyperplane
in $\mathbb P^n_{\mathbb C}$, and let $a_0x_0+\dots+a_nx_n$ be a linear form
whose vanishing determines $H$.  Then, for all $P\in\mathbb P^n_{\mathbb C}$,
choose homogeneous coordinates $[x_0:\dots:x_n]$ for $P$, and let
$$\lambda_H(P)
  = -\frac12\log\frac{|a_0x_0+\dots+a_nx_n|^2}
    {(|a_0|^2+\dots+|a_n|^2)(|x_0|^2+\dots+|x_n|^2)}\;.$$
Again, this does not depend on the choice of the form $a_0x_0+\dots+a_nx_n$
or on the choice of homogeneous coordinates $[x_0:\dots:x_n]$.

Then the theorem of Cartan to be used here is as follows (see \cite{R4},
for example).

\begin{theorem}\label{cartan}
Let $n$ be a positive integer,
let $H_1,\dots,H_q$ be hyperplanes in $\mathbb P^n_{\mathbb C}$,
and let $f\colon\mathbb C\to \mathbb P^n_{\mathbb C}$
be a holomorphic curve whose image is not contained in a hyperplane.  Then, for any $\epsilon>0$,
$$\int_0^{2\pi} \max_J \sum_{j\in J}
    \lambda_{H_j}(f(re^{i\theta}))\,\frac{d\theta}{2\pi}
  \le_\exc (n+1+\epsilon)T_f(r),$$
where $J$ varies over the same collection of sets as in Theorem \ref{schmidt},
and where the notation $\le_\exc$ means that the inequality holds for all
$r\in(0,\infty)$ outside of a set of finite Lebesgue measure.
\end{theorem}

\begin{remark}\label{remk_schmidt_cartan}
In the above two theorems, there is a finite union $Z_1$ of proper linear
subspaces of $\mathbb P^n_k$ or $\mathbb P^n_{\mathbb C}$
(in Theorems \ref{schmidt} and \ref{cartan}, respectively), depending only
on the hyperplanes $H_1,\dots,H_q$, with the following properties.
In Theorem \ref{schmidt}, the subset $Z$ may be taken to be the union of $Z_1$
and a finite union of points, and in Theorem \ref{cartan}, the condition on the
holomorphic curve $f$ may be relaxed to allow any nonconstant holomorphic map
$f\colon\mathbb C\to\mathbb P^n_{\mathbb C}$ whose image is not contained
in $Z_1$.  See \cite{vojta_aj} and \cite{vojta_smt_ajm_1997}, respectively.
\end{remark}

\begin{remark}
Since the functions $\lambda_{H,\upsilon}$ are bounded from below, we may
assume in Theorems \ref{schmidt} and \ref{cartan} that $\bigcap H_i=\emptyset$
(include more hyperplanes).  Then, in each of these theorems, we may also
require that all of the sets $J$ have $n+1$ elements.
\end{remark}

When working with the methods of Corvaja, Zannier, Evertse, Ferretti, Autissier,
and the first author, it is useful to phrase Theorems \ref{schmidt}
and \ref{cartan} in terms of divisors in a linear system.
This leads to Theorems \ref{schmidt_base} and \ref{cartan_base}, below.
(In this paper, they are used to prove the General Theorems of
Section \ref{sect_aut}, and to prove Proposition \ref{b_thmc_prop21}.)

Theorems \ref{schmidt_base} and \ref{cartan_base} use the notion of
Weil functions (the functions $\lambda_H$ above are examples of Weil functions).
The reader not already familiar with these functions is encouraged to
refer to Section \ref{sect_weil} or to the references contained therein
for more information.

\begin{theorem}\label{schmidt_base}
Let $k$ be a number field, let $S$ be a finite set of places of $k$ containing
all archimedean places, let $X$ be a complete variety over $k$, let $D$ be
a Cartier divisor on $X$,  let $V$ be a nonzero linear subspace of
$H^0(X,\mathscr O(D))$, let $s_1,\dots,s_q$ be nonzero elements of $V$,
let $\epsilon>0$, and let $c\in\mathbb R$.
For each $i=1,\dots,q$, let $D_j$ be the Cartier divisor $(s_j)$,
and let $\lambda_{D_j}$ be a Weil function for $D_j$.
Then there is a proper Zariski-closed subset $Z$ of $X$, depending only
on $k$, $S$, $X$, $L$, $V$, $s_1,\dots,s_q$, $\epsilon$, $c$, and the
choices of Weil and height functions, such that the inequality
\begin{equation}\label{ineq_schmidt_base}
  \sum_{\upsilon\in S}\max_J\sum_{j\in J}\lambda_{D_j,\upsilon}(x)
    \le (\dim V+\epsilon)h_D(x) + c
\end{equation}
holds for all $x\in(X\setminus Z)(k)$.  Here the set $J$ ranges
over all subsets of $\{1,\dots,q\}$ such that the sections $(s_j)_{j\in J}$
are linearly independent.
\end{theorem}

\begin{proof}
Let $d=\dim V$.  We may assume that $d>1$ (otherwise, all $D_j$ are
the same divisor, and the sets $J$ have at most one element each,
so (\ref{ineq_schmidt_base}) follows immediately from (the number theory
version of) the First Main Theorem).

Let $\Phi\colon X\dashrightarrow\mathbb P^{d-1}_k$ be
the rational map associated to the linear system $V$.  Let $X'$ be the
closure of the graph of $\Phi$, and let $p\colon X'\to X$ and
$\phi\colon X'\to\mathbb P^{d-1}_k$ be the projection morphisms.

Note that, even though $\Phi$ extends to the morphism
$\phi\colon X'\to\mathbb P^{d-1}_k$, the linear system
of $H^0(X',p^{*}\mathscr O(D))$ corresponding to $V$ may still have
base points.  What is true, however,
is that there is an effective Cartier divisor $B$ on $X'$ such that,
for each nonzero $s\in V$, there is a hyperplane $H$ in $\mathbb P^{d-1}_k$
such that $p^{*}(s)-B=\phi^{*}H$.
(More precisely, $\phi^{*}\mathscr O(1)\cong\mathscr O(p^{*}D-B)$.  The map
$$\alpha\colon H^0(X',\mathscr O(p^{*}D-B)) \to H^0(X,\mathscr O(p^{*}D))$$
defined by tensoring with the canonical global section $1_B$ of $\mathscr O(B)$
is injective, and its image contains $p^{*}(V)$.
The preimage $\alpha^{-1}(p^{*}(V))$ corresponds to a base-point-free
linear system for the divisor $p^{*}D-B$.)

For each $j=1,\dots,q$ let $H_j$ be the hyperplane in $\mathbb P^{d-1}_k$
for which $p^{*}(s_j)-B=\phi^{*}H_j$.  Choose a Weil function $\lambda_B$
for $B$.  Then, for all $\upsilon\in S$ and all $j=1,\dots,q$, we have
$$p^{*}\lambda_{D_j,\upsilon}
  = \phi^{*}\lambda_{H_j,\upsilon} + \lambda_{B,\upsilon} + O(1)\;.$$
Therefore it will suffice to prove that for any $c'\in\mathbb R$ the inequality
\begin{equation}\label{eq_schmidt_base_1}
  \sum_{\upsilon\in S}\max_J\sum_{j\in J}
    (\lambda_{H_j,\upsilon}(\phi(x))+\lambda_{B,\upsilon}(x))
  \le (\dim V+\epsilon)h_D(p(x)) + c'
\end{equation}
holds for all $x\in X'(k)$ outside of some proper Zariski-closed subset $Z'$
of $X'$.  Indeed, for suitable $c'$ this will imply (\ref{ineq_schmidt_base})
outside of $Z:=p(Z'\cup\Supp B)$.  The set $Z$ is Zariski-closed in $X$
because $p\colon X'\to X$ is a proper morphism, and $Z\ne X$ because $p$
is birational.

For any subset $J$ of $\{1,\dots,q\}$, the sections $s_j$, $j\in J$ are
linearly independent elements of $V$ if and only if the hyperplanes $H_j$,
$j\in J$ lie in general position in $\mathbb P^{d-1}_k$.  Therefore
we may apply Theorem \ref{schmidt}, to obtain that for any $c''$ the inequality
\begin{equation}\label{eq_schmidt_base_2}
  \sum_{\upsilon\in S}\max_J\sum_{j\in J} \lambda_{H_j,\upsilon}(\phi(x))(x)
    \le (\dim V+\epsilon)h(\phi(x)) + c''
\end{equation}
holds for all $x\in X'(k)$ for which $\phi(x)$ does not lie in a finite union
$Z''$ of proper linear subspaces of $\mathbb P^{d-1}_k$.  Here $Z''$ depends
on $k$, $S$, $d$, $H_1,\dots,H_q$, $\epsilon$, and $c''$, but not on $x$.

Since each set $J$ as above has at most $\dim V$ elements and $B$ is effective,
we have
$$(\#J)\lambda_{B,\upsilon}(x) \le (\dim V)\lambda_{B,\upsilon}(x) + O(1)$$
for all $x\in X'(k)$.  Therefore, (\ref{eq_schmidt_base_2}) implies
(\ref{eq_schmidt_base_1}) for all $x\in X'(k)$ outside of $Z':=\phi^{-1}(Z'')$.
Since the coordinates of $\phi$ are associated to linearly independent elements
of $p^{*}(V)$, the (closed) set $Z'$ is not all of $X'$.
\end{proof}

\begin{theorem}\label{cartan_base}
Let $X$ be a complex projective variety,  let $D$ be
a Cartier divisor on $X$, let $V$ be a nonzero linear subspace of
$H^0(X,\mathscr O(D))$, and let $s_1,\dots,s_q$ be nonzero elements of $V$.
For each $i=1,\dots,q$, let $D_j$ be the Cartier divisor $(s_j)$,
and let $\lambda_{D_j}$ be a Weil function for $D_j$.
Let $f\colon\mathbb C\to X$ be a holomorphic map with Zariski-dense image.
Then, for any $\epsilon>0$,
$$
  \int_0^{2\pi} \max_J\sum_{j\in J}\lambda_{D_j}(f(re^{i\theta}))
    \le_\exc (\dim V)T_{f,D}(r) + O(\log^{+} T_{f,D}(r)) + o(\log r)\;.
$$
Here the set $J$ ranges over all subsets of $\{1,\dots,q\}$ such that
the sections $(s_j)_{j\in J}$ are linearly independent.
\end{theorem}

The proof of this theorem is very similar to the proof of
Theorem \ref{schmidt_base}, and is omitted.

\begin{remark}
In Theorems \ref{schmidt_base} and \ref{cartan_base}, if the rational map
$X\dashrightarrow\mathbb P^{d-1}$ is generically finite, then there is
a proper Zariski-closed subset $Z_1$ of $X$, depending only on $D$, $V$,
and $s_1,\dots,s_q$, with the following properties.
In Theorem \ref{schmidt_base}, the subset $Z$ may be taken to be the union
of $Z_1$ and a finite union of points, and in Theorem \ref{cartan_base},
the condition on the holomorphic curve $f$ may be relaxed to allow any
nonconstant holomorphic map $f\colon\mathbb C\to\mathbb P^n_{\mathbb C}$
whose image is not contained in $Z_1$.  Indeed, in the notation of the
proof of Theorem \ref{schmidt_base}, we may take
$Z_1=p(\phi^{-1}(Z_1')\cup Z_2\cup\Supp B)$,
where $Z_1'\subseteq\mathbb P^{d-1}$ is the closed subset of
Remark \ref{remk_schmidt_cartan} and $Z_2$ is the subset of $X'$
where $\phi$ is not finite.
\end{remark}

\begin{remark}
In Theorems \ref{schmidt_base} and \ref{cartan_base}, we may assume
(by shrinking $V$ or using more sections) that $s_1,\dots,s_q$ span $V$.
Under this assumption, we may instead take the maximum over all sets $J$
such that $(s_j)_{j\in J}$ is a basis of $V$.
\end{remark}

\section{Weil Functions}\label{sect_weil}

Let $X$ be a variety over $\mathbb C$, let $D$ be an effective Cartier divisor
on $X$, and let $s=1_D$ be a canonical section of $\mathscr O(D)$ (i.e.,
a global section for which $(s)=D$).  Choose a smooth metric $|\cdot|$
on $\mathscr O(D)$.  In Nevanlinna theory, one often encounters the
function
\begin{equation}\label{eq_weil_intro}
  \lambda_D(x) := -\log|s(x)|\;;
\end{equation}
this is a real-valued function on $X(\mathbb C)\setminus\Supp D$.
It is linear in $D$ (over a suitable domain), so by linearity and continuity
it can be extended to a definition of $\lambda_D$ for a general Cartier divisor
$D$ on $X$.

Weil functions are counterparts to such functions in number theory.
In number theory, however, the situation is more complicated because
a number field has infinitely many inequivalent absolute values, and
because $\overline k_v$
is not locally compact unless $v$ is archimedean.  One often needs to
sum over all absolute values of $k$, so an inequality that is true up
to $O(1)$ in Nevanlinna theory needs to be true up to a constant that
vanishes for almost all (all but finitely many) $v$ in number theory.

Giving a full definition of Weil functions is beyond the scope of this
paper; see Lang \cite[Ch.~10]{lang} for a full treatment of Weil
functions, or the second author \cite[Sect.~8]{vojta_cm} for a brief treatment
with a few more details than are given here.

\medskip
\emph{Throughout this section, $k$ is a number field, and varieties and
morphisms are implicitly taken to be over $k$.}
\medskip

\begin{definition}
Let $X$ be a variety.
\begin{enumerate}
\item[(a).]  The set $X(M_k)$ is defined as
$$X(M_k) = \coprod_{v\in M_k} X(\overline k_v)\;.$$
Here $k_v$ is the completion of $k$ at a place $v$, and $\overline k_v$ is
an algebraic closure of $k_v$.  By abuse of notation, if $Z$ is a subset of
$X$, then we write $X(M_k)\setminus Z$ to mean $(X\setminus Z)(M_k)$.
The topology on $X(\overline k_v)$ is the topology determined by the (unique)
extension of $v$ to $\overline k_v$.  A subset of $X(M_k)$ is open if and
only if its intersection with $X(\overline k_v)$ is open for all $v\in M_k$.
\item[(b).]  An \textbf{$M_k$-constant} is a collection $c=(c_v)_{v\in M_k}$
of real numbers such that $c_v=0$ for almost all $v$.
\item[(c).]  Let $D$ be a Cartier divisor on $X$.  A \textbf{Weil function}
for $D$ is a function $\lambda_D\colon X(M_k)\setminus\Supp D\to\mathbb R$
that satisfies the following condition.  For each Zariski-open subset
$U$ of $X$ and each function $f\in K(X)^{*}$ such that the divisors $D$
and $(f)$ coincide on $U$, there is a continuous locally $M_k$-bounded
function $\alpha\colon U(M_k)\to\mathbb R$ such that
\begin{equation}\label{weil_def_eq}
  \lambda_D(x) = -\log\|f(x)\|_v + \alpha(x)
\end{equation}
for all $v\in M_k$ and all $x\in U(\overline k_v)\setminus\Supp D$.
\item[(d).]  Let $D$ be a Cartier divisor on $X$, and let $v$ be a place
of $k$.  A \textbf{local Weil function} for $D$ at $v$ is a function
$\lambda_{D,v}\colon X(\overline k_v)\setminus\Supp D\to\mathbb R$
that satisfies the following condition.  For each Zariski-open subset
$U$ of $X$ and each function $f\in K(X)^{*}$ such that the divisors $D$
and $(f)$ coincide on $U$, there is a locally bounded continuous function
$\alpha\colon U(\overline k_v)\to\mathbb R$ such that
$$ \lambda_{D,v}(x) = -\log \|f(x)\|_v + \alpha(x) $$
for all $x\in U(\overline k_v)\setminus\Supp D$.
(Local Weil functions are also called \emph{local heights.})
\end{enumerate}
\end{definition}

For the definition of locally $M_k$-bounded function, see
Lang \cite[Ch.~10, Sect.~1]{lang}.  The details of this definition are not
really necessary for this paper, though.  For this paper, it is enough to know
that if $\alpha\colon U(M_k)\to\mathbb R$ is a locally $M_k$-bounded function,
then its restriction to $U(\overline k_v)$ for any $v$ is locally bounded
and continuous.  Conversely, any locally bounded continuous function
$U(\overline k_v)\to\mathbb R$ extends to a locally $M_k$-bounded function
on $U(M_k)$.

Because of this, if $D$ is a Cartier divisor on a variety $X$,
if $\lambda_D$ is a Weil function for $D$, and if $v$ is a place of $k$,
then $\lambda_{D,v}:=\lambda_D\big|_{X(\overline k_v)\setminus\Supp D}$
is a local Weil function for $D$ at $v$.  Conversely, given any finite
collection $S$ of places of $k$ and local Weil functions for $D$
at all $v\in S$, there is a Weil function $\lambda_D$ for $D$
such that $\lambda_{D,v}$ coincides with the given local Weil functions
at all $v\in S$.  In this paper, we will only use local Weil functions at
finitely many places, so (global) Weil functions are not strictly speaking
necessary here.  They are only used because using local Weil functions
would lead to the question of whether the condition depends on the choice
of place of $k$.  (Perhaps we could have defined a
\emph{semilocal Weil function} to be a collection of local Weil functions
at finitely many places $v$.  One could replace global Weil functions
throughout this paper with such semilocal Weil functions without any loss
of rigor.)

Given a complete variety $X$ and a Cartier divisor $D$ on $X$, there exists
a Weil function for $D$.  For a proof when $X$ is projective, see
\cite[Ch.~10, Thm.~3.5]{lang}.
The proof in the general case is sketched in \cite[Thm.~8.7]{vojta_cm}.

Existence of local Weil functions is much easier.  For local Weil functions
at an archimedean place $v$, we have $\overline k_v=\mathbb C$, so the method
of (\ref{eq_weil_intro}) works.
If $v$ is non-archimedean, then one can make similar constructions using
the notion of $v$-adically metrized line sheaf due to S. Zhang
\cite[Appendix]{zhang_jams}.  Or, see the sketched proof in
\cite[Thm.~8.7]{vojta_cm}.

By \cite[Ch.~10, Lemma~1.4]{lang}, if $U$ is a nonempty
open subset of $X$, then $U(M_k)$ is dense in $X(M_k)$ and $U(\overline k_v)$
is dense in $X(\overline k_v)$ for all $v$.  (Lang works with $\overline k$
instead of $\overline k_v$, but the proofs are the same.)
Therefore, if $D$ is a Cartier divisor on a variety $X$, and if $U$ is any
nonempty subset of $X$ disjoint from $\Supp D$, any function
$\lambda\colon U(M_k)\to\mathbb R$ that satisfies (\ref{weil_def_eq})
extends uniquely to a Weil function for $D$.  This fact will be used frequently
without specific mention.

If $D$ is an effective divisor, then a Weil function $\lambda_D$ for $D$ may be
extended to a function $\lambda_D\colon X(M_k)\to\mathbb R\cup\{+\infty\}$.

\begin{proposition}\label{weil_props}
Weil functions have the following properties.  Let $X$ be a variety.
\begin{enumerate}
\item[(a).]  \textbf{Additivity.}  If $\lambda_1$ and $\lambda_2$ are
Weil functions for Cartier divisors $D_1$ and $D_2$ on $X$, respectively,
then $\lambda_1+\lambda_2$ extends uniquely to a Weil function for $D_1+D_2$.
\item[(b).]  \textbf{Functoriality.}  If $\lambda$ is a Weil function for
a Cartier divisor $D$ on $X$, and if $f\colon Y\to X$ is a morphism of
$k$-varieties such that $f(Y)\nsubseteq\Supp D$,
then $x\mapsto\lambda(f(x))$ is a Weil function for
the Cartier divisor $f^{*}D$ on $Y$.
\item[(c).]  \textbf{Normalization.}  If $X=\mathbb P^n_k$, and if $D$ is the
hyperplane at infinity, then the function
$$ \lambda_{D,v}([x_0:\dots:x_n])
  := -\log\frac{\|x_0\|_v}{\max\{\|x_0\|_v,\dots,\|x_n\|_v\}} $$
is a Weil function for $D$.  For the definition of $\|\cdot\|_v$,
see \cite[Sect.~1]{vojta_cm}
\end{enumerate}
Local Weil functions (at a given place $v$) also have these properties.
\end{proposition}

For proofs of these properties, see Lang \cite[Ch.~10]{lang}.

Next we show which Weil functions correspond to effective Cartier divisors
(in more generality than will be needed here), and use that result to derive
a uniqueness result (Corollary \ref{weil_uniq}).

\begin{remark}\label{remk_eff}
On a normal variety, the Cartier divisors are exactly
the locally principal Weil divisors \cite[II~6.11.2]{Hartshorne}, and
a Cartier divisor is effective as a Cartier divisor if and only if it is
effective as a Weil divisor \cite[II~6.3A]{Hartshorne}.
\end{remark}

\begin{proposition}\label{weil_eff}
Let $X$ be a normal complete variety, let $D$ be a Cartier divisor on $X$,
and let $\lambda_D$ be a Weil function for $D$.  Then the following
conditions are equivalent.
\begin{enumerate}
\item $D$ is effective.
\item $\lambda_D$ is bounded from below by an $M_k$-constant.
\item for all $v\in M_k$, $\lambda_{D,v}$ is bounded from below.
\item there exists a $v\in M_k$ such that $\lambda_{D,v}$ is bounded from below.
\end{enumerate}
\end{proposition}

\begin{proof}
For the implication (i)$\implies$(ii), see
\cite[Ch.~10, Prop.~3.1]{lang}.
The implications (ii)$\implies$(iii)$\implies$(iv) are trivial.

To show that (iv)$\implies$(i), assume that $D$ is not effective.
By Remark \ref{remk_eff}, $D$ (as a Weil divisor) has at least one component
with negative multiplicity.  Fix a closed point $x\in X$ that lies on
that prime divisor, but not on any other irreducible
component of $\Supp D$, and let $U$ be a Zariski-open neighborhood of $x$.
Take a sequence $(x_n)$ of points in $U(\overline k_v)\setminus\Supp D$ that
converges to $x$ in the $v$-topology.  Then the sequence $\lambda_{D,v}(x_n)$
goes to $-\infty$; thus $\lambda_{D,v}$ is not bounded from below.
\end{proof}

\begin{corollary}\label{weil_uniq}
Let $X$ be a complete variety (not necessarily normal), let $D$ be a Cartier
divisor on $X$, and let $\lambda_1$ and $\lambda_2$ be Weil functions for $D$.
Then there are $M_k$-constants $c$ and $c'$ such that
$$-c \le \lambda_1-\lambda_2 \le c'\;.$$
\end{corollary}

\begin{proof}
By Proposition \ref{weil_props}a we may assume that $D=0$ and $\lambda_2=0$.
The corollary is then immediate from Proposition \ref{weil_eff}
and Proposition \ref{weil_props}b (pull back to the normalization).
\end{proof}

The following proposition is useful for constructing Weil functions
via the ``min-max'' method.  It will also be used later in this paper.

\begin{proposition}\label{weil_min}
Let $X$ be a variety; let $D_1,\dots,D_n$ and $D$ be Cartier divisors on $X$;
and let $\lambda_{D_1},\dots,\lambda_{D_n}$ be Weil functions for
$D_1,\dots,D_n$, respectively.  Assume that $D_i-D$ is effective for all $i$,
and that
$$\bigcap_{i=1}^n \Supp(D_i-D) = \emptyset\;.$$
Then the function
$$\lambda_D := \min\{\lambda_{D_i}:i=1,\dots,n\}$$
is a Weil function for the divisor $D$.
\end{proposition}

\begin{proof}
See Lang \cite[Ch.~10, Prop.~3.2]{lang}.
\end{proof}

A slightly more general statement will be used to rephrase
Definition \ref{bidef2}.

\begin{lemma}\label{weil_max_fancy}
Let $X$ be a projective variety, and let $U_1,\dots,U_n$ be Zariski-open
subsets of $X$ that cover $X$.  Let $D_1,\dots,D_n$ be Cartier divisors on $X$
such that $D_i\big|_{U_i}$ is effective for all $i$, and let $\lambda_{D_i}$
be Weil functions for $D_i$ for all $i$.  Then there is an $M_k$-constant
$\gamma=(\gamma_v)$ such that, for all $v$ and all $x\in X(\overline k_v)$
there is an $i$ such that $x\in U_i$ and $\lambda_{D_i,v}(x)\ge \gamma_v$.
\end{lemma}

\begin{proof}
By taking a finite refinement of $\{U_i\}$, we may assume that each $U_i$ is
affine, and that $D_i\big|_{U_i}=(f_i)$ for some nonzero
$f_i\in\mathscr O(U_i)$ for all $i$.

Then $\lambda_{D_i}$ is locally $M_k$-bounded from below on $U_i(M_k)$
for all $i$.  Indeed, by \cite[Ch.~10, Prop.~1.3]{lang},
the function $-\log\|f_i\|$ is locally $M_k$-bounded from below on $U_i(M_k)$;
therefore so is $\lambda_{D_i}$ since by definition of Weil function
and N\'eron function the difference between the two functions is
locally $M_k$-bounded.

By \cite[Ch.~10, Prop.~1.2]{lang}, there are
affine $M_k$-bounded sets $E_1,\dots,E_n$ such that $\bigcup E_i=X(M_k)$ and
such that $E_i\subseteq U_i(M_k)$ for all $i$.

By definition of locally $M_k$-bounded function
(see \cite[Ch.~10, Sect.~1]{lang}), for each $i$
there is an $M_k$-constant $\gamma_i$ such that $\lambda_{D_i}\ge\gamma_i$
on $E_i$.  This concludes the proof, with
$\gamma=\min\{\gamma_1,\dots,\gamma_n\}$.
\end{proof}

\begin{remark}
It is possible to prove Proposition \ref{weil_min} using
Lemma \ref{weil_max_fancy} rather easily.
\end{remark}

\section{The General Theorem}\label{sect_aut}

In this section, we prove the General Theorem (in both the arithmetic
and analytic cases) by using the filtration constructed by
Autissier in \cite{Aut2}. Before presenting the proof, we first discuss some consequences of the General Theorem, in addition
to Theorems C and D discussed in the introduction.

\subsection{The equi-degree case}
Using the General Theorem, we  can derive  a more precise (if not sharp)
theorem of Schmidt's subspace theorem type (in the arithmetic case)
and Nevanlinna's Second Main Theorem type (in the complex case)
for the divisors which are (only) assumed to have equi-degree.

\begin{definition}
Suppose that X is a complete variety of dimension $n$.  Let
$D_1,\dots,D_q$ be effective Cartier divisors on $X$, and let
$D=D_1+D_2+\cdots+D_q$. We
say that \textbf{$D$ has equi-degree with respect to $D_1, D_2, \dots, D_q$} if
$D_i\idot D^{n-1}=\frac{1}{q}D^{n}$ for all $i=1,\dots,q$.
\end{definition}

The important result associated to the concept of  equi-degree is as follows.

\begin{lemma}[Levin {\cite[Lemma~9.7]{levin_annals}}]\label{lem_lev09_9_7}
Let X be a complete variety of dimension $n$.  If $D_j, 1 \leq j \leq q$,
are big and nef Cartier divisors, then there exist positive real numbers
$r_j$ such that $D=\sum_{j=1}^q r_j D_j$ has equi-degree.
\end{lemma}

(In \cite{levin_annals}, Levin assumed that $X$ is projective,
but his proof works more generally for complete varieties without change.)

Since divisors $r_jD_j$ and $D_j$ have the same support, the above lemma  tells us that we can always make the given  big and
nef divisors have equi-degree without changing their supports.
This means that the  condition of equi-degree, rather than the previous
assumptions of linear or numerical
equivalence for the divisors $D_1,\cdots,D_q$,  is indeed
a correct (or reasonable) assumption in the study of (quasi) hyperbolicity
as well as the study of the degeneracy of integral points on the complement
$X\backslash D$.

To compute $\gamma(D_j)$ for $j=1, \dots, q$, we use the following lemma
of Autissier.

\begin{lemma}[{\cite[Lemma 4.2]{aut}}] \label{lem_aut1_4_2}
Suppose $E$ is a big and base-point free Cartier divisor on a projective
variety $X$, and $F$ is a nef Cartier divisor on $X$ such that
$F-E$ is also nef. Let $\beta>0$ be a positive real number. Then, for any
positive integers  $N$ and $m$  with $1 \leq m\leq \beta N$, we have
\begin{eqnarray*}
h^0(NF-mE)&\ge& {F^n\over n!}N^n
- {F^{n-1}\idot E\over (n-1)!}N^{n-1}m \\
&~&+ {(n-1)F^{n-2}\idot E^2\over n!} N^{n-2}\min\{m^2, N^2\} +O(N^{n-1})
\end{eqnarray*}
where the implicit constant depends on $\beta$.
\end{lemma}
We now compute $\sum_{m\ge 1} h^0(ND-mD_i)$ for  each $1 \leq i\leq q$.
Let $n=\dim X$, and assume that $n\ge2$.
Let $\beta={D^n\over n D^{n-1}\idot D_i}$ and $A=(n-1)D^{n-2}\idot D_i^2$.
Then, by Lemma \ref{lem_aut1_4_2},
\begin{equation*}
 \begin{split}
  &\sum_{m=1}^{\infty} h^0(ND-mD_i) \label{aut1} \\
  &\quad \ge \sum_{m=1}^{[\beta N]}
    \left( {D^n\over n!}N^n
    - {D^{n-1}\idot D_i\over (n-1)!}N^{n-1}m + {A\over n!}N^{n-2}\min\{m^2, N^2\}\right)+O(N^n) \\
  &\quad\ge \left( {D^n\over n!} \beta
    - {D^{n-1}\idot D_i\over (n-1)!}{\beta^2\over 2} + {A\over n!}g(\beta)\right)N^{n+1}+O(N^n) \\
  &\quad= \left({\beta\over 2}+{A\over D^n}g(\beta)\right) D^n {N^{n+1}\over n!} +O(N^n) \\
  &\quad= \left({\beta\over 2}+\alpha \right) N h^0(ND) +O(N^n)
 \end{split}
\end{equation*}
 where $\alpha:= {A\over D^n}g(\beta)$ and
  $g: \mathbb R^+\rightarrow \mathbb R^+$ is the function given by $g(x)={x^3\over 3}$ if $x\leq 1$ and $g(x)=x-{2\over 3}$ for $x\ge 1$.
 Now from the assumption of equi-degree,   $D_i\idot D^{n-1}={1\over q} D^n$,
so
 $\beta ={q\over n}$. Hence
  $$\gamma(D_i)=
\inf_N  {Nh^0(ND)\over \sum_{m\ge 1} h^0(ND-mD_i)}\leq {1\over {\beta\over 2}+\alpha }< {2n\over q}.$$

The General Theorem in Section \ref{intro} thus gives

\begin{theorem} [Arithmetic Part]\label{thm_equideg_ar}
Let $k$ be a number field and let $S\subseteq M_k$ be a finite set
containing all archimedean places.
Let $X$ be a projective variety of dimension $\ge 2$ over $k$,
and let $D_1,\dots,D_q$ be ample Cartier divisors on $X$
that intersect properly.
Assume that $D:=\sum_{j=1}^q D_j$ has equi-degree respect to $D_1,\dots,D_q$.
Then, for $\epsilon_0>0$ small enough  (which only depends on the given divisors), the inequality
$$m_S(x, D) < \left({2\dim X\over q}-\epsilon_0\right) h_D(x)$$
holds for all $k$-rational points $x\in X(k)$ outside of a proper
Zariski-closed subset of $X$.
\end{theorem}

\begin{theorem}[Analytic Part]\label{thm_equideg_an}
Let $X$ be a
complex projective variety of dimension $\ge 2$, and let $D_1, \dots, D_q$ be
ample Cartier divisors on $X$ that intersect properly.
Assume that $D:=\sum_{j=1}^q D_j$
has equi-degree with respect to $D_1, \dots, D_q$.
Let $f\colon \mathbb C\rightarrow X$ be a holomorphic map with Zariski-dense
image.
Then, for $\epsilon_0>0$ small enough,
$$m_f(r, D) <_\exc \left({2\dim X\over q}-\epsilon_0\right)T_{f,D}(r).$$
\end{theorem}

In general, we can prove

\begin{theorem}[Arithmetic Part]\label{thm_ban_ar}
Let $k$ be a number field and let $S\subseteq M_k$ be a finite set
containing all archimedean places.
Let $X$ be a projective variety of dimension $\ge 2$ over $k$,
and let $D_1, \dots, D_q$
be effective, big, and nef Cartier divisors on $X$ that intersect properly.
Let $r_i>0$ be real numbers such that $D:=\sum_{i=1}^q r_i D_i$ has
equi-degree (such numbers exist due to Lemma \ref{lem_lev09_9_7}).
Then, for $\epsilon_0>0$ small enough,  the inequality
$$ \sum_{j=1}^q r_j m_S(x, D_j)  <  \left({2\dim X\over q}-\epsilon_0\right)
\left(\sum_{j=1}^q r_jh_{D_j}(x)\right)$$
holds for all $x\in X(k)$ outside a proper Zariski-closed subset of $X$.
\end{theorem}

\begin{theorem}[Analytic Part]\label{thm_ban_an}
Let $X$ be a complex projective variety of dimension $\ge 2$,
and let $D_1, \dots, D_q$ be effective, big, and nef Cartier
divisors on $X$ that intersect properly.
Let $r_i>0$ be real numbers such that $D:=\sum_{i=1}^q r_i D_i$ has
equi-degree (such numbers exist due to Lemma \ref{lem_lev09_9_7}).
Let $f\colon \mathbb C\rightarrow X$ be a holomorphic map with Zariski-dense
image.
Then, for $\epsilon_0>0$ small enough,
$$ \sum_{j=1}^q r_j m_f(r, D_j) <_\exc \left({2\dim X\over q}-\epsilon_0\right) \left(\sum_{j=1}^q r_jT_{f, D_j}(r)\right).$$
\end{theorem}

The proofs of Theorems \ref{thm_ban_ar} and \ref{thm_ban_an} are similar to
the proofs of Theorems \ref{thm_equideg_ar} and Theorem \ref{thm_equideg_an}
above, except in this case, the divisor
 $D=\sum_{i=1}^q r_i D_i$ is an $\mathbb R$-divisor, so we need to approximate this divisor by
a $\mathbb Q$-divisor $\widehat{D}:=\sum_{i=1}^q a_i D_i$, similar to the argument carried out in \cite{R4} and \cite{ru}.
We omit the proofs here.  We note that Theorems \ref{thm_ban_ar}
and \ref{thm_ban_an} greatly improve
the earlier results in \cite{R4} and \cite{ru}.

\subsection{The proof of the General Theorem}
The proof of the General Theorems stated in Section \ref{intro} uses the the filtration constructed by
Pascal Autissier (see \cite{Aut2}). We first review his results.

Let $D_1, \dots, D_r$ be effective Cartier divisors on a projective variety $X$.
Assume that they intersect properly on $X$, and
that $\bigcap_{i=1}^r D_i$ is non-empty.  Let $\mathscr L$ be a line sheaf over $X$ with $l:=h^0(\mathscr L)\ge 1$.

\begin{definition} A subset $N\subset \mathbb N^r$ is said to be
\textbf{saturated} if ${\bf a}+{\bf b}\in N$ for any
${\bf a}\in  \mathbb N^r$ and ${\bf b}\in N$.
\end{definition}

\begin{lemma}[{\cite[Lemma 3.2]{Aut2}}]\label{lemm_aut2_3_2}
Let $A$ be a local ring and $(\phi_1, \dots, \phi_r)$ be a regular sequence of
$A$. Let $M$ and $N$ be two saturated subsets of $\mathbb N^r$.Then
$${\mathcal I}(M)\cap {\mathcal I}(N)={\mathcal I}(M\cap N),$$
where, for $N\subset \mathbb N^r$, ${\mathcal I}(N)$ is the ideal of $A$ generated by
$\{\phi_1^{b_1}\cdots \phi_q^{b_r}~|~{\bf b}\in N\}$.
\end{lemma}

\begin{remark}\label{remk_aut2_3_2}
We use Lemma \ref{lemm_aut2_3_2} in the following particular situation:
Let
$$\bigtriangleup
  = \{{\bf t}=(t_1, \dots, t_r)\in (\mathbb R^+)^r ~|~ t_1+\cdots +t_r=1\}.$$
For each ${\bf t}\in \bigtriangleup$ and $x\in \mathbb R^+$, let
$$N({\bf t}, x)=\{{\bf b}\in \mathbb N^r~|~t_1b_1+\cdots +t_rb_r\ge x\}.$$
Notice that $N({\bf t}, x)\cap N({\bf u}, y)\subset N(\lambda {\bf t}+(1-\lambda){\bf u}, \lambda x+ (1-\lambda)y)$ for all $\lambda\in[0,1]$.
So, from Lemma \ref{lemm_aut2_3_2}, we have
\begin{equation}
{\mathcal I}(N({\bf t}, x))\cap {\mathcal I}(N({\bf u}, y))\subset
{\mathcal I}(N(\lambda {\bf t}+(1-\lambda){\bf u}, \lambda x+ (1-\lambda)y))
\end{equation}
for any ${\bf t}, {\bf u}\in \bigtriangleup$; $x, y\in  \mathbb R^+$;
and $\lambda\in[0,1]$.
\end{remark}

\begin{definition} Let $W$ be a vector space of finite dimension.
A \textbf{filtration} of $W$ is a family of subspaces
${\mathcal F}=({\mathcal F}_x)_{x\in \mathbb R^+}$ of subspaces of
$W$ such that $\mathcal F_x\supseteq\mathcal F_y$ whenever $x\le y$,
and such that ${\mathcal F}_x=\{0\}$ for $x$ big enough.
A basis ${\mathcal B}$ of $W$
is said to be \textbf{adapted to $\mathcal F$} if
${\mathcal B}\cap {\mathcal F}_x$ is a basis of ${\mathcal F}_x$ for every real number $x\ge 0$.
\end{definition}

\begin{lemma}[Levin \cite{levin_annals}, Autissier \cite{Aut2}]\label{lemm_filt}
Let ${\mathcal F}$ and ${\mathcal G}$ be two filtrations of $W$.
Then there exists a basis
of $W$ which is adapted to both ${\mathcal F}$ and ${\mathcal G}$.
\end{lemma}

For any fixed ${\bf t}\in \bigtriangleup$, we construct a filtration of $H^0(X, \mathscr L)$ as follows:
for  $x\in  \mathbb R^+$,
one defines the ideal ${\mathcal I}({\bf t}, x)$ of ${\mathscr O}_X$ by
\begin{equation}\label{def_I}
  {\mathcal I}({\bf t}, x)
    = \sum_{{\bf b}\in N({\bf t}, x)} {\mathscr O}_X(-\sum_{i=1}^r b_iD_i)\;,
\end{equation}
and let
\begin{equation}\label{def_filtr_h_0}
  {\mathcal F}({\bf t})_x
    = H^0(X, {\mathcal I}({\bf t}, x)\otimes \mathscr L)\;.
\end{equation}
Then $({\mathcal F}({\bf t})_x)_{x\in \mathbb R^+}$ is a filtration of $H^0(X, \mathscr L)$.

For $s\in H^0(X, \mathscr L) -\{0\}$, let $\mu_{\bf t}(s)=\sup\{y\in \mathbb R^+~|~s\in {\mathcal F}({\bf t})_y\}.$
Also let
\begin{equation}\label{def_F}
  F({\bf t})
    = {1\over h^0(\mathscr L)}\int_0^{+\infty}(\dim{\mathcal F}({\bf t})_x)\,dx\;.
\end{equation}

\begin{remark}\label{remk_levin_aut}
Let ${\mathcal B}=\{s_1, \dots, s_l\}$ be a basis of $H^0(X, \mathscr L)$ with $l=h^0(\mathscr L)$. Then we have
$$F({\bf t})\ge {1\over l}\int_0^{\infty} \#({\mathcal F}({\bf t})_x\cap {\mathcal B})dx={1\over l}\sum_{k=1}^l \mu_{{\bf t}}(s_k),$$
where equality holds if  ${\mathcal B}$ is adapted to the filtration $({\mathcal F}({\bf t})_x)_{x\in \mathbb R^+}.$
\end{remark}
The key result we will use about this filtration is the following Proposition.

\begin{proposition}[{\cite[Th\'eor\`eme 3.6]{Aut2}}]\label{aut2_thm_3_6}
With the notations and assumptions above, let
$F: \bigtriangleup \rightarrow \mathbb R^+$ be the map defined in (\ref{def_F}). Then $F$
is concave. In particular, for ${\bf t}\in  \bigtriangleup$,
\begin{equation}\label{aut_ineq}
  F({\bf t})
    \ge \min_i \left({1\over h^0(\mathscr L)}\sum_{m\ge 1} h^0(\mathscr L(-mD_i))\right)\;.
\end{equation}
\end{proposition}
We include a proof here for the sake of completeness.
\begin{proof} For any ${\bf t}, {\bf u}\in \bigtriangleup$ and $\lambda \in [0, 1]$, we need to prove that
\begin{equation}F(\lambda {\bf t}+(1-\lambda){\bf u})\ge \lambda F({\bf t})+(1-\lambda)F({\bf u}).\end{equation}
By Lemma \ref{lemm_filt}, there exists a basis ${\mathcal B}=\{s_1, \dots, s_l\}$ of $H^0(X,\mathscr L)$ with $l=h^0(\mathscr L)$,
which is adapted both to $({\mathcal F}({\bf t})_x)_{x\in \mathbb R^+}$
and to $({\mathcal F}({\bf u})_y)_{y\in \mathbb R^+}$.
For $x, y\in \mathbb R^+$, by Lemma \ref{lemm_aut2_3_2} (or
Remark \ref{remk_aut2_3_2}), since $D_1, \dots, D_r$ intersect properly on $X$
$${\mathcal F}({\bf t})_x\cap {\mathcal F}({\bf u})_y
\subset
{\mathcal F}(\lambda {\bf t}+(1-\lambda){\bf u})_{\lambda x+ (1-\lambda)y}.$$
For $s\in H^0(X, \mathscr L)-\{0\}$, we have, from the definition of $\mu_{{\bf t}}(s)$ and $\mu_{{\bf u}}(s)$,
$s\in {\mathcal F}(\lambda {\bf t}+(1-\lambda){\bf u})_{\lambda x+ (1-\lambda)y}$
for $x<\mu_{{\bf t}}(s)$ and $y<\mu_{{\bf u}}(s)$, and thus
$$\mu_{\lambda {\bf t}+(1-\lambda){\bf u}}(s)\ge \lambda  \mu_{{\bf t}}(s) +(1-\lambda)\mu_{{\bf u}}(s).$$
Taking $s=s_j$ and summing it over $j=1, \dots, l$, we get,
by Remark \ref{remk_levin_aut},
$$F(\lambda {\bf t}+(1-\lambda){\bf u})\ge \lambda {1\over l}\sum_{j=1}^l  \mu_{{\bf t}}(s_j) +(1-\lambda){1\over l}\sum_{j=1}^l
\mu_{{\bf u}}(s_j).$$
On the other hand, since ${\mathcal B}=\{s_1, \dots, s_l\}$ is a basis
adapted to both ${\mathcal F}({\bf t})$ and ${\mathcal F}({\bf u})$,
from Remark \ref{remk_levin_aut},
$F({\bf t})={1\over l}\sum_{j=1}^l \mu_{{\bf t}}(s_j)$
and $F({\bf u})={1\over l}\sum_{j=1}^l \mu_{{\bf u}}(s_j)$.
Thus $$F(\lambda {\bf t}+(1-\lambda){\bf u})\ge \lambda F({\bf t})+(1-\lambda)F({\bf u}),$$
 which proves that  $F$ is a convex function.

To prove (\ref{aut_ineq}), let ${\bf e}_1=(1, 0, \dots, 0)$, $\cdots$,
${\bf e}_r=(0, 0, \dots, 1)$ be the natural basis of $\mathbb R^r$,
and write, for ${\bf t} \in \bigtriangleup$, ${\bf t}=t_1{\bf e_1}+\cdots+t_r{\bf e}_r$.
Then, notice that $t_1+\cdots +t_r=1$, from the convexity of $F$, we get
$$F({\bf t}) = F(t_1{\bf e_1}+\cdots+t_r{\bf e}_r) \ge (t_1+\cdots +t_r) \min_i F({\bf e}_i)
= \min_i F({\bf e}_i)$$
and, obviously, $F({\bf e}_i)={1\over  h^0(\mathscr L)}\sum_{m\ge 1} h^0(\mathscr L(-mD_i))$ for $i=1, \dots, r$.
\end{proof}

We are now ready to prove the General Theorems.
We first consider the arithmetic case.

Let $D_1, \dots, D_q$ be effective Cartier divisors intersecting properly
on $X$, and let $D=D_1+\cdots+D_q$.
Choose Weil functions $\lambda_{D_1},\dots,\lambda_{D_q}$ for the Cartier
divisors $D_1,\dots,D_q$, respectively.  Since these divisors are effective,
we may assume that $\lambda_{D_i,\upsilon}(y)\ge0$
for all $i=1,\dots,q$; all $\upsilon\in S$; and all
$y\in X(\overline k_\upsilon)\setminus\Supp D_i$
(this can be done by Proposition \ref{weil_eff}).

Let $\epsilon>0$, and pick a positive integer $N$ such that
\begin{equation}\label{choice_of_N}
  \max_{1\le j\le q} \frac{Nh^0(\mathscr L^N)}{\sum_{m\ge1} h^0(\mathscr L^N(-mD_j))}
    < \max_{1\le j\le q} \gamma(\mathscr L, D_j)+{\epsilon\over 4}\;.
\end{equation}

\begin{lemma}[similar to {\cite[Lemma~20.7]{vojta_cm}}] \label{vojta_cm_20_7}
Let $X$ be a complete variety over a number field $k$;
let $D_1,\dots,D_q$ be effective Cartier divisors on $X$; let $D=D_1+\dots+D_q$;
and let $\lambda_{D_1},\dots,\lambda_{D_q}$ be Weil functions for
$D_1,\dots,D_q$, respectively.  Let
$$\Sigma
  = \left\{\sigma\subseteq \{1,\dots,q\}
    \bigm| \bigcap_{j\in \sigma} \Supp D_j\ne\emptyset\right\}\;.$$
For each  $\sigma \in \Sigma$,
write
\begin{equation}D=D_{\sigma, 1} + D_{\sigma, 2}\;, \end{equation}
where
$$D_{\sigma, 1} = \sum_{j\in \sigma} D_j \qquad\text{and}\qquad
  D_{\sigma, 2} = \sum_{j\notin \sigma} D_j \;.$$
Let $\lambda_{D_{\sigma,2}}=\sum_{j\notin\sigma}\lambda_{D_j}$
for all $\sigma\in\Sigma$.
Then there exists an $M_k$-constant  $(C_{\upsilon})_{\upsilon \in M_k}$,
depending only on $X$; $D_1,\dots,D_q$; and the chosen Weil functions,
such that
$$\min_{\sigma\in \Sigma}\lambda_{D_{\sigma, 2}, \upsilon}(x)\leq C_{\upsilon}$$
for all $x\in X(\overline k_{\upsilon})$ and all $\upsilon \in M_k$.
\end{lemma}

\begin{proof}
The  definition of the set $\Sigma$  implies that
$$\bigcap_{\sigma\in \Sigma} \Supp D_{\sigma, 2}=\emptyset,$$
because for all $x\in X$ the set
$\sigma = \{j\in\{1,\dots,q\} \mid x\in \Supp D_j\}$ is
an element in $\Sigma$, and then $x\not\in \Supp D_{\sigma, 2}$.
The lemma
then follows from Proposition \ref{weil_min} and Corollary \ref{weil_uniq},
since $\Sigma$ is a finite set.
\end{proof}

Let $c\ge 1$ be an integer such that $h^0(\mathscr L^N(-cD_j))=0$
for $j=1, \dots, q$ and fix an integer $b$ with
$b\ge {cn\over N\epsilon_0}$, where $\epsilon_0>0$ is chosen such that
$$\epsilon_0 <{\epsilon\over ( \max_{1\le j\le q} \gamma(\mathscr L, D_j)+1+\epsilon)( 4\max_{1\le j\le q} \gamma(\mathscr L, D_j)+1+\epsilon)}.$$

 For $\sigma\in \Sigma$, let
$$\bigtriangleup_{\sigma}=\left\{{\bf a}=(a_i)\in {\mathbb N}^{\#\sigma}~|~\sum_{i\in \sigma} a_i=b\right\}.$$
For ${\bf a}\in \bigtriangleup_{\sigma}$ (hence ${1\over b}{\bf a}\in  \bigtriangleup$), as
 above, one defines (see (\ref{def_I}),  (\ref{def_filtr_h_0}), and (\ref{def_F}))
the ideal ${\mathcal I}(x)$ of ${\mathscr O}_X$ by
$${\mathcal I}(x)
  = \sum_{{\bf b}} {\mathscr O}_X\left(-\sum_{i\in \sigma} b_iD_{i}\right)$$
  where the sum is taken for all
  ${\bf b}\in {\mathbb N}^{\#\sigma}$ with $\sum_{i\in \sigma} a_ib_i\ge bx$.
 Let
$${\mathcal F}(\sigma; {\bf a})_x
  = H^0(X, \mathscr L^N\otimes  {\mathcal I}(x))\;,$$
which we regard as a subspace of $H^0(X,\mathscr L^N)$, and let
$$F(\sigma; {\bf a})
    = {1\over h^0(\mathscr L^N)}\int_0^{+\infty}(\dim{\mathcal F}(\sigma;{\bf a})_x)\,dx\;.$$
  Applying Proposition \ref{aut2_thm_3_6} with the line sheaf being taken as $\mathscr L^N$, we have
$$F(\sigma; {\bf a})\ge  \min_{1\leq i \leq q}\left({1\over h^0(\mathscr L^N)}\sum_{m\ge 1} h^0(\mathscr L^N(-mD_i))\right).$$
Let $\mathcal B_{\sigma; {\bf a}}$ be a basis of
$H^0(X, \mathscr L^N)$ adapted to the above filtration
$\{{\mathcal F}(\sigma; {\bf a})_x\}_{x\in\mathbb R^+}$.
By Remark \ref{remk_levin_aut},
$F(\sigma, {\bf a}) = {1\over h^0(\mathscr L^N)}\sum_{s\in \mathcal B_{\sigma; {\bf a}}} \mu(s)$, where
 $\mu(s)$  is the largest rational number
for which $s\in {\mathcal F}(\sigma; {\bf a})_{\mu}$. Hence
\begin{equation}\label{sum_mu_geq}
 \sum_{s\in \mathcal B_{\sigma; {\bf a}}} \mu(s)
    \ge \min_{1\leq i\leq q} \sum_{m\ge 1} h^0(\mathscr L^N(-mD_i))\;.
\end{equation}
It is important to note that the set $\bigcup_{\sigma; {\bf a}}  \mathcal B_{\sigma; {\bf a}}$ is a finite set.

 Let
$$\lambda_D = \sum_{j=1}^q \lambda_{D_j}
  \qquad\text{and}\qquad
  \lambda_{D_{\sigma,1}} = \sum_{j\in\sigma} \lambda_{D_j}
  \quad\text{for all $\sigma\in\Sigma$}$$
(note that $\lambda_{D_{\sigma,2}}$ was defined already in
Lemma \ref{vojta_cm_20_7}).

Given $\sigma\in\Sigma$ and $\mathbf a\in\bigtriangleup_\sigma$, any
nonzero section $s\in H^0(X,\mathscr L^N)$ can be written locally as
$$s = \sum_{\mathbf b} f_{\mathbf b}\prod_{i\in\sigma} 1_{D_i}^{b_i}\;,$$
where $f_{\mathbf b}$ is a local section of
$\mathscr L^N(-\sum_{i\in\sigma} b_iD_i)$,
$1_{D_i}$ is the canonical section of $\mathscr O(D_i)$ for each $i$,
and the sum is taken for all $\mathbf b\in\mathbb N^{\#\sigma}$
with $\sum_{i\in\sigma} a_ib_i\ge b\mu(s)$.  Moreover, $f_{\mathbf b}=0$
for all but finitely many $\mathbf b$.

By a compactness argument, there exist a finite covering
$\{U_j\}_{j\in J_{\sigma,\mathbf a,s}}$ of $X$ by Zariski-open sets and
a finite set $K=K_{\sigma,\mathbf a,s}\subseteq\mathbb N^{\#\sigma}$ such that
\begin{equation}\label{s_local_sum}
  s = \sum_{\mathbf b\in K} f_{s,j;\mathbf b} \prod_{i\in\sigma} 1_{D_i}^{b_i}
\end{equation}
on $U_j$ for all $j\in J_{\sigma,\mathbf a,s}$,
where $f_{s,j;\mathbf b}\in\Gamma(U_j,\mathscr L^N(-\sum_{i\in\sigma} b_iD_i))$
and all $\mathbf b\in K$ satisfy $\sum_{i\in\sigma} a_ib_i\ge b\mu(s)$.

\begin{lemma}\label{lemm_bound_weil_s}
For each $\upsilon\in S$ there is a constant $c_v$ (which depends also
on $\epsilon$ and the choices of Weil functions $\lambda_s$)
with the following property.
Pick $\sigma\in\Sigma$, $\mathbf a\in\bigtriangleup_{\sigma}$, and
$s\in\mathcal B_{\sigma;\mathbf a}$, and choose a Weil function
$\lambda_s$ for the divisor $(s)$.  Then, for each $\upsilon\in S$,
\begin{equation}\label{lambdainq}
  \lambda_{s,\upsilon}(P)
    \ge \min_{\mathbf b} \sum_{i\in\sigma}
      b_i\lambda_{D_i, \upsilon}(P)+c_\upsilon
\end{equation}
for all $P\in X(k)$,
where the minimum is taken over all $\mathbf b\in K$ as in (\ref{s_local_sum}).
\end{lemma}

\begin{proof}
For all $\sigma\in\Sigma$, all $\mathbf a\in\bigtriangleup_\sigma$,
all $s\in\mathcal B_{\sigma;\mathbf a}$, all $j\in J_{\sigma,\mathbf a,s}$,
and all $\upsilon\in S$ there are constants $c_{\upsilon; j, \sigma,\mathbf a}$
such that, for all $P$ in a $\upsilon$-bounded subset of $U_j(k_\upsilon)$,
\begin{equation}\label{lambdainq_pre}
  \lambda_{s,\upsilon}(P)
    \ge \min_{\mathbf b} \sum_{i\in \sigma}
      b_i\lambda_{D_i, \upsilon}(P)+c_{\upsilon;j,\sigma,\mathbf a}\;.
\end{equation}
These bounded subsets can be chosen to cover all of $X(k_\upsilon)$.
This gives (\ref{lambdainq}), since the sets $\Sigma$,
$\bigtriangleup_\sigma$, $\mathcal B_{\sigma;\mathbf a}$,
and $J_{\sigma,\mathbf a,s}$ are finite.

Note also that (\ref{lambdainq_pre}) for all $P\in X(\overline k)$
follows from Proposition \ref{prop_aut_weil_ineq}.
\end{proof}

We may further assume that the sets $K=K_{\sigma,\mathbf a,s}$ in
Lemma \ref{lemm_bound_weil_s} contain no $\mathbf b$ for which
$f_{s,j;\mathbf b}=0$ for all $s$ and $j$.  In particular,
since $H^0(X,\mathscr L^N(-cD_j))=0$ for all $j$, we may assume that
all $\mathbf b\in K$ satisfy $b_j<c$ for all $j$.
Also, from the assumption that  $D_1, \dots, D_q$ intersect properly
(and hence lie in general position), we have $\#\sigma_{P, \upsilon}\leq n$.
Therefore, by the choice of the integer $b$, we may assume that
all $\mathbf b\in K$ in (\ref{lambdainq}) satisfy
\begin{equation}\label{mathbf_b_ineq}
  \sum_{i\in\sigma} b_i \le nc \le bN\epsilon_0\;.
\end{equation}

Now for $P\in X(k)$ and $\upsilon\in S$ pick $\sigma_{P,\upsilon}\in\Sigma$
in Lemma \ref{vojta_cm_20_7} such that
\begin{equation}
  \lambda_{D_{\sigma_{P, \upsilon},2}, \upsilon}(P)\leq C_{\upsilon}\label{b}
\end{equation}
where $C_{\upsilon}$ is the $M_k$-constant appearing in
Lemma \ref{vojta_cm_20_7}, which depends only on $X$; $D_1, \dots, D_q$;
and the chosen Weil functions.

For $i\in \sigma_{P, \upsilon}$ we let
\begin{equation}\label{def_t_Pvi}
t_{P, \upsilon; i}={\lambda_{D_{i}, \upsilon}(P)\over \sum_{j\in \sigma_{P, \upsilon}} \lambda_{D_{j}, \upsilon}(P)}.
\end{equation}
Note that $\sum_{i\in \sigma_{P, \upsilon}} t_{P, \upsilon; i}=1$.  Choose ${\bf a}_{P, \upsilon} =(a_{P, \upsilon; i}) \in \bigtriangleup_{\sigma_{P, \upsilon}}$
such that
\begin{equation}\label{condition_aPv}
  |b t_{P, \upsilon; i} -a_{P, \upsilon; i}|\leq 1
    \qquad\text{for all $i\in \sigma_{P, \upsilon}$.}
\end{equation}
Using (\ref{lambdainq}) with $\sigma=\sigma_{P, \upsilon}$
and ${\bf a}= {\bf a}_{P, \upsilon}$, (\ref{def_t_Pvi}), (\ref{condition_aPv}),
and (\ref{mathbf_b_ineq}), we get, for any $s\in \mathcal B_{\sigma; {\bf a}}$,
\begin{equation}\label{lambda_s_chain}
\begin{split}
  \lambda_{s, \upsilon}(P)
    &\ge \min_{\mathbf b\in K} \sum_{i\in \sigma_{P, \upsilon}}
          b_i\lambda_{D_i, \upsilon}(P)+O_{\upsilon}(1) \\
    &= \left( \sum_{j\in \sigma_{P, \upsilon}} \lambda_{D_j, \upsilon}(P)  \right)
      \min_{\mathbf b\in K} \sum_{i\in \sigma_{P, \upsilon}}  b_it_{P, \upsilon; i}  +O_{\upsilon}(1) \\
    &\ge \left( \sum_{j\in \sigma_{P, \upsilon}} \lambda_{D_j, \upsilon}(P)  \right)
      \min_{\mathbf b\in K} \sum_{i\in \sigma_{P, \upsilon}}  b_i {a_{P, \upsilon; i}-1\over b} +O_{\upsilon}(1) \\
     & \ge  (\mu(s) - N\epsilon_0)  \left( \sum_{j\in \sigma_{P, \upsilon}} \lambda_{D_j, \upsilon}(P)  \right)  +O_{\upsilon}(1),
\end{split}
\end{equation}
where the set $K$ is as in (\ref{s_local_sum}).
Therefore, by (\ref{lambda_s_chain}), (\ref{sum_mu_geq}),
the definition of $\lambda_{D_{\sigma,1}}$, and (\ref{b}), we have
\begin{equation}\label{weil_chain_ineq}
\begin{split}
\sum_{s\in \mathcal B_{P,\sigma}}\lambda_{s, \upsilon} (P)
  &\ge \left(\sum_{s\in \mathcal B_{P,\sigma}} (\mu(s) - N\epsilon_0)\right) \left( \sum_{i\in \sigma_{P, \upsilon}} \lambda_{D_i, \upsilon}(P)  \right)  +O_{\upsilon}(1) \\
  &\ge \left( \min_{1\leq i\leq q} \sum_{m\ge 1} h^0(\mathscr L^N(-mD_i))  -Nl\epsilon_0 \right) \left( \sum_{i\in \sigma_{P, \upsilon}} \lambda_{D_i, \upsilon}(P)  \right)  +O_{\upsilon}(1) \\
 &= \left( \min_{1\leq i\leq q} \sum_{m\ge 1} h^0(\mathscr L^N(-mD_i))-Nl\epsilon_0 \right)  \lambda_{D_{\sigma_{P, \upsilon}, 1},\upsilon}(P)+O_{\upsilon}(1) \\
&=  \left( \min_{1\leq i \leq q}\sum_{m\ge 1} h^0(\mathscr L^N(-mD_i))-Nl\epsilon_0 \right)  \lambda_{D, \upsilon}(P) +O_{\upsilon}(1)\;,
\end{split}
\end{equation}
where $l=h^0(\mathscr L^N)$.

For any basis ${\mathcal B}$ of $H^0(X, {\mathscr L}^N)$, we recall
the notation $(\mathcal B)$ from (\ref{def_parens_B}):
$$(\mathcal B) = \sum_{s\in\mathcal B} (s).$$
For such $\mathcal B$, choose a Weil function $\lambda_{\mathcal B}$
for the divisor $(\mathcal B)$; then (\ref{weil_chain_ineq}) gives
\begin{equation}\label{weilbase}
  \max_{\sigma; {\bf a}} \lambda_{ \mathcal B_{\sigma; {\bf a}}}
    \ge \left( \min_{1\leq i \leq q}\sum_{m\ge 1} h^0(\mathscr L^N(-mD_i))-Nl\epsilon_0 \right)  \lambda_{D, \upsilon} + O_v(1)\;.
\end{equation}
Thus from the definition of $\Nevbir(\mathscr L, D)$ (see Definition \ref{bidef1}) and by  (\ref{choice_of_N}),  we get
$$\Nevbir(\mathscr L, D) \leq \max_{1\leq j\leq q} \gamma(\mathscr L, D_j)$$
and thus the General Theorem follows from Theorem \ref{b_thmd}.  This finishes the proof.

Note that we can  continue the proof without using the notion of $\Nevbir(\mathscr L, D_j)$  and Theorem \ref{b_thmd}.
Indeed write
$$\bigcup_{\sigma; {\bf a}}  \mathcal B_{\sigma; {\bf a}}    = \mathcal B_1\cup\cdots \cup \mathcal B_{T_1}=\{s_1, \dots, s_{T_2}\}.$$
 For each $i=1,\dots, T_1$, let $J_i\subseteq\{1,\dots,T_2\}$ be the subset
such that $\mathcal B_i = \{s_j:j\in J_i\}$.  Then, by (\ref{weilbase}),
for each $v\in S$,
\begin{equation}\label{weilbase2}
  \begin{split}
  \left( \min_{1\leq i \leq q}\sum_{m\ge 1} h^0(\mathscr L^N(-mD_i))-Nl\epsilon_0 \right)  \lambda_{D, \upsilon}
    &\le \max_{1\le i\le T_1} \lambda_{\mathcal B_i,v} + O_{\upsilon}(1)\\
    &= \max_{1\le i\le T_1} \sum_{j\in J_i} \lambda_{s_j,v} + O_{\upsilon}(1).
  \end{split}
\end{equation}
By Theorem \ref{schmidt_base} with  $\epsilon$ in Theorem \ref{schmidt_base} taken as ${\epsilon\over  4\max_{1\le j\le q} \gamma(\mathscr L, D_j)+1+\epsilon}$,
there  is a proper Zariski-closed subset $Z$ of $X$ such that the inequality
\begin{equation}\label{schmidt_ineq}
  \sum_{v\in S} \max_J \sum_{j\in J} \lambda_{s_j,v}(x)
    \le \left(l +{\epsilon\over  4\max_{1\le j\le q} \gamma(\mathscr L, D_j)+1+\epsilon}\right)h_{ND}(x)
\end{equation}
holds for all $x\in X(k)$ outside of $Z$;
here $l=h^0(\mathscr L^N)$ and the maximum is taken over all subsets $J$ of $\{1,\dots,T_2\}$
for which the sections $s_j$, $j\in J$, are linearly independent.

Combining (\ref{weilbase2}) and (\ref{schmidt_ineq}) gives $$\sum_{\upsilon\in S} \lambda_{D, \upsilon}(x) \leq {l+{\epsilon\over 4\max_{1\le j\le q} \gamma(\mathscr L, D_j)+1+\epsilon}\over \min_{1\leq i \leq q}\sum_{m\ge 1} h^0(\mathscr L^N(-mD_i))-Nl\epsilon_0}
h_{\mathscr L^N}(x)$$
for all $x\in X(k)$ outside of $Z$.  Here we used the fact that all of
the $J_i$ occur among the $J$ in (\ref{schmidt_ineq}).
Using  (\ref{choice_of_N}),  the fact that $l=h^0(\mathscr L^N)$ and  $h_{\mathscr L^N}(x)=Nh_{\mathscr L}(x)$, we have, for $x\in X(k)$ outside of a proper Zariski-closed subset $Z$ of $X$,
\begin{equation*}
\begin{split}
\sum_{\upsilon\in S} \lambda_{D, \upsilon}(x) &\leq {( \max_{1\le j\le q} \gamma(\mathscr L, D_j) +{\epsilon\over 4})
(1+{\epsilon\over 4\max_{1\le j\le q} \gamma(\mathscr L, D_j)+\epsilon})\over
1-( \max_{1\le j\le q} \gamma(\mathscr L, D_j) +{\epsilon\over 4})\epsilon_0}
h_{D}(x)\\
&\leq (\max_{1\le j\le q} \gamma(\mathscr L, D_j) +\epsilon)h_D(x),
\end{split}
\end{equation*}
from our choice of $\epsilon_0$.
This proves the General Theorem for the arithmetic case.

The proof in the analytic case is similar, by replacing
Theorem \ref{schmidt_base} (Schmidt's Subspace Theorem) with
Theorem \ref{cartan_base} (H. Cartan's theorem).

\begin{remark}
With notation as in Definition \ref{def_aut_lambda}, let
\begin{equation}
  \gamma'(\mathscr L, D)
    = \inf_{N,V} \frac{N\dim V}
      {\sum_{m\ge 1} \dim(V\cap H^0(X,\mathscr L^N(-mD)))}\;,
\end{equation}
where the infimum is over all positive integers $N$ and all linear subspaces
$V$ of $H^0(X,\mathscr L^N)$, and the fraction is taken to be $+\infty$
if the denominator is zero.  Here we identify $H^0(X,\mathscr L^N(-mD))$
with a subspace of $H^0(X,\mathscr L^N)$ via the injection gotten by
tensoring with $1_D^m$.
Then the General Theorems remain true when $\gamma$ is replaced by $\gamma'$.
The same proof works with minimal changes, but we have not included it
here because we know of no applications so far.
\end{remark}

\section{The Birational Nevanlinna Constant}\label{sect_nevbir}

The goal of this section and the next two sections is to prove that
the two definitions (Definitions \ref{bidef1} and \ref{bidef2})
of the birational Nevanlinna constant $\Nevbir(D)$
given in Section \ref{intro} are equivalent.

We first recall the definition of the Nevalinna constant $\Nev(D)$ introduced
by the first author (Definition \ref{def_nevintro}), as well as the important
results regarding $\Nev(D)$.

\begin{definition}[see \cite{R4} and \cite{ru}] \label{def_nev}
Let $X$ be a normal projective variety, and let $D$ be an effective
Cartier divisor on $X$.  The \textbf{Nevanlinna constant} of $D$,
denoted $\Nev(D)$, is given by
\begin{equation}\label{Nconstant}
  \Nev(D) = \inf_{N,V,\mu} {\dim V\over \mu}\;,
\end{equation}
where the infimum is taken over all triples $(N,V,\mu)$ such that
$N$ is a positive integer, $V$ is a linear subspace of $H^0(X,\mathscr L^N)$
with $\dim V\ge2$, and $\mu$ is a positive real number such that,
for all $P\in \Supp D$, there exists a basis $\mathcal B$ of $V$ with
\begin{equation} \label{large}
  \sum_{s\in \mathcal B} \ord_E (s) \ge \mu \ord_E (ND)
\end{equation}
for all irreducible components $E$ of $D$ passing through $P$.
If $\dim H^0(X,\mathscr O(ND))\leq 1$ for all positive integers $N$,
we define $\Nev(D)=+\infty$.  For a general complete variety $X$,
$\Nev(D)$ is defined by pulling back to the normalization of $X$.
\end{definition}

Next, we rephrase the definition of $\Nev(D)$.
Recall from (\ref{def_parens_B}) that if $\mathcal B$ is a finite set
of global sections of a line sheaf $\mathscr L$ on a variety $X$,
then $(\mathcal B)$ denotes the divisor
$$(\mathcal B) = \sum_{s\in\mathcal B} (s)\;.$$

First of all, notice that the divisor $D$ plays two separate roles in
Definition \ref{def_nev}:
the role in (\ref{large}), and the space $H^0(X,\mathscr O(ND))$.  At times it will be
convenient to separate those two roles, so we introduce a line sheaf
$\mathscr L$ and replace $H^0(X,\mathscr O(ND))$ with $H^0(X,\mathscr L^N)$.

Next, we can impose (\ref{large}) on all divisor components passing
through $P$ (not just those occurring in $D$).
Also, since we are taking the infimum in Definition \ref{def_nev}, we can
require $\mu$ to be rational.

With these two changes, the condition on the triple is equivalent to
requiring that the $\mathbb Q$-divisor $(\mathcal B)-\mu ND$ be effective
near $P$.

This leads to the following definition.

\begin{definition}\label{def_nev2}
Let $X$ be a complete variety, let $D$ be an effective Cartier divisor
on $X$, and let $\mathscr L$ be a line sheaf on $X$.  If $X$ is normal, then
we define
$$\Nev(\mathscr L,D) = \inf_{N,V,\mu} \frac{\dim V}{\mu}\;.$$
Here the inf is taken over all triples $(N,V,\mu)$ such that
$N\in\mathbb Z_{>0}$, $V$ is a linear subspace of
$H^0(X,\mathscr L^N)$ with $\dim V>1$,
and $\mu>0$ is a rational number, that satisfy the following property.
For all $P\in X$ there is a basis $\mathcal B$ of $V$ such that
\begin{equation}\label{eq2}
  (\mathcal B) \ge \mu ND
\end{equation}
in a Zariski-open neighborhood $U$ of $P$,
relative to the cone of effective $\mathbb Q$-divisors on $U$.
If there are no such triples $(N,V,\mu)$, then $\Nev(\mathscr L,D)$ is defined
to be $+\infty$. For a general complete variety $X$, $\Nev(\mathscr L,D)$
is defined by pulling back to the normalization of $X$.
\end{definition}

Note that $\Nev(\mathscr O(D),D)$ (as defined above) coincides with $\Nev(D)$
as in Definition \ref{def_nev}.

To prove the equivalence, we start by recalling the second definition of
$\Nevbir(D)$ (Definition \ref{bidef2}) given in Section \ref{intro}:

\begin{definition}[see Definition \ref{bidef2}] \label{def_nevprime}
Let $X$ be a complete variety, let $D$ be an effective Cartier divisor
on $X$, and let $\mathscr L$ be a line sheaf on $X$.  If $X$ is normal, then
we define
$$\Nevbir(\mathscr L,D) = \inf_{N,V,\mu} \frac{\dim V}{\mu}\;,$$
where the infimum passes over all triples $(N,V,\mu)$ such that
$N\in\mathbb Z_{>0}$, $V$ is a linear subspace of
$H^0(X,\mathscr L^N)$
with $\dim V>1$, and $\mu\in\mathbb Q_{>0}$, with the following property.
There are a variety $Y$ and a proper birational morphism $\phi\colon Y\to X$
such that the following condition holds.
For all $Q\in Y$ there is a basis $\mathcal B$ of $V$ such that
\begin{equation}\label{eq3}
  \phi^{*}(\mathcal B) \ge \mu N\phi^{*}D
\end{equation}
in a Zariski-open neighborhood $U$ of $Q$,
relative to the cone of effective $\mathbb Q$-divisors on $U$.
If there are no such triples $(N,V,\mu)$, then $\Nevbir(\mathscr L,D)$
is defined
to be $+\infty$. For a general complete variety $X$, $\Nevbir(\mathscr L,D)$
is defined by pulling back to the normalization of $X$.

If $L$ is a Cartier divisor or Cartier divisor class on $X$, then
we define $\Nevbir(L,D)=\Nevbir(\mathscr O(L),D)$.  We also define
$\Nevbir(D)=\Nevbir(D,D)$.
\end{definition}

(Note that a \emph{birational morphism} from $X$ to $Y$ is a morphism $X\to Y$
that has an inverse as a rational map; in other words, it is a rational map
$X\dashrightarrow Y$ that is defined everywhere on $X$.)

\begin{remark}\label{remk_D_linear}
It is easy to see from the definitions that if $n$ is a positive integer then
$$\Nev(\mathscr L,nD) = n\Nev(\mathscr L,D)
  \qquad\text{and}\qquad
  \Nevbir(\mathscr L,nD) = n\Nevbir(\mathscr L,D)\;.$$
\end{remark}

We have the following easy comparison.

\begin{lemma}\label{lemma_comp_nevs}
Let $X$, $D$, and $\mathscr L$ be as in Definitions \ref{def_nev2}
and \ref{def_nevprime}.  Then
$$\Nevbir(\mathscr L,D) \le \Nev(\mathscr L,D)\;.$$
\end{lemma}

\begin{proof}
We may assume that $X$ is normal, because in both Definition \ref{def_nev2}
and Definition \ref{def_nevprime} the general case is handled by
pulling back to the normalization.

If a triple $(N,V,\mu)$ satisfies the condition of Definition \ref{def_nev2},
then it also satisfies the condition of Definition \ref{def_nevprime},
because in the latter condition we can take $Y=X$ and let $\phi$ be the
identity map.  Thus, the infimum in Definition \ref{def_nevprime}
is being taken over a larger set.
\end{proof}

The only difference between the definitions is that in
$\Nevbir(\mathscr L,D)$, the basis $\mathcal B$ is allowed to be taken
locally on a blowing-up of $X$, rather than on $X$ itself.

We now show that $\Nevbir$ can be viewed as a \emph{birationalization}
of $\Nev$.

\begin{proposition}\label{prop_b_nev}
Let $X$ be a complete variety over a number field $k$, let $D$ be an effective
Cartier divisor on $X$, and let $\mathscr L$ be a line sheaf on $X$.
Then:
\begin{enumerate}
\item[(a).]  $\Nevbir(\mathscr L,D)$ is a birational invariant, in the sense
that if $\phi\colon Y\to X$ is a model of $X$, then
$$\Nevbir(\phi^{*}\mathscr L,\phi^{*}D) = \Nevbir(\mathscr L,D)\;.$$
\item[(b).]  For all $\epsilon>0$ there is a model $\phi\colon Y\to X$ of $X$
such that
$$\Nev(\phi^{*}\mathscr L,\phi^{*}D) < \Nevbir(\mathscr L,D) + \epsilon\;.$$
\item[(c).]  In particular,
$$\Nevbir(\mathscr L,D)
  = \inf_{\phi\colon Y\to X} \Nev(\phi^{*}\mathscr L,\phi^{*}D)\;,$$
where the infimum is over all models of $X$.
\end{enumerate}
\end{proposition}

\begin{proof}
(a).  This is clear from Definition \ref{def_nevprime}, since the condition
on $(N,V,\mu)$ remains true if $Y$ is replaced by another variety $Y'$
that dominates $Y$ (i.e., there is a proper birational morphism $Y'\to X$
that factors through $Y$).

(b).  Let $(N,V,\mu)$ be a triple satisfying the condition of
Definition \ref{def_nevprime}, for which
$$\frac{\dim V}{\mu} < \Nevbir(\mathscr L,D) + \epsilon\;,$$
and let $\phi\colon Y\to X$ be as in Definition \ref{def_nevprime}.
Then the triple  $(N,\phi^{*}V,\mu)$ satisfies the condition of
Definition \ref{def_nev2} for $\Nev(\phi^{*}\mathscr L,\phi^{*}D)$.
This proves (b).

(c).  This part is immediate from part (b) and from Lemma \ref{lemma_comp_nevs}.
\end{proof}

To conclude this section, we prove the following Proposition which would lead to  the equivalence of the two definitions of
$\Nevbir(D)$.

\begin{proposition}\label{prop_nevprime_weil}
Let $X$ be a normal complete variety over a number field, let $D$ be
an effective Cartier divisor
on $X$, let $\mathscr L$ be a line sheaf on $X$, let $V$ be a linear subspace
of $H^0(X,\mathscr L)$ with $\dim V>1$, and let $\mu>0$ be a rational number.

Consider the following conditions.
\begin{enumerate}
\item There are a variety $Y$ and a proper birational morphism
$\phi\colon Y\to X$ such that for all $Q\in Y$ there is a basis $\mathcal B$
of $V$ such that
$$ \phi^{*}(\mathcal B) \ge \mu \phi^{*}D $$
in a Zariski-open neighborhood $U$ of $Q$, relative to the cone of
effective $\mathbb Q$-divisors on $U$.
\item There are finitely many bases $\mathcal B_1,\dots,\mathcal B_\ell$
of $V$; Weil functions $\lambda_{\mathcal B_1},\dots,\lambda_{\mathcal B_\ell}$
for the divisors $(\mathcal B_1),\dots,(\mathcal B_\ell)$, respectively;
a Weil function $\lambda_D$ for $D$; and an $M_k$-constant $c$ such that
\begin{equation}
  \max_{1\le i\le\ell}\lambda_{\mathcal B_i} \ge \mu\lambda_D - c
\end{equation}
(as functions $X(M_k)\to\mathbb R\cup\{+\infty\}$).
\end{enumerate}
If (i) is true, then so is (ii).
\end{proposition}

\begin{proof}
Assume that (i) holds.

By quasi-compactness, we may assume that only finitely many open subsets $U$
occur in (i).  Let $U_1,\dots,U_\ell$ be a collection of such subsets.
Since the condition on $\mathcal B$ only depends on $U$, we may fix a
basis $\mathcal B_i$ for each open subset $U_i$.  Also, for each $i$
let $\lambda_{\mathcal B_i}$ be a Weil function for the divisor
$(\mathcal B_i)$, and let $\lambda_D$ be a Weil function for $D$.

Fix a positive integer $n$ such that $n\mu\in\mathbb Z$, and such that
the divisor $n\phi^{*}(\mathcal B_i)-n\mu\phi^{*}D$ is effective on $U_i$
for all $i$.  The above Weil functions can be pulled back to give Weil functions
for $\phi^{*}(\mathcal B_i)$ and $\phi^{*}D$, respectively, on $Y$.

By Lemma \ref{weil_max_fancy} applied to the divisors
$D_i:=n\phi^{*}(\mathcal B_i)-n\mu\phi^{*}D$ for all $i$ and to
the open sets $U_1,\dots,U_\ell$, there is an $M_k$-constant $\gamma$ such that
$$\max_{i=1,\dots,\ell}
    \bigl(n\phi^{*}\lambda_{\mathcal B_i}-n\mu\phi^{*}\lambda_D\bigr)
  \ge\gamma\;.$$
Therefore (ii) holds, with $c=-\gamma/n$.
\end{proof}

The converse will be proved in the next section
(Proposition \ref{prop_mu_b_growth_equivs}).

\section{Models of Varieties, b-divisors, and b-Weil Functions}\label{models}

In light of the birational nature of $\Nevbir$, it is useful to consider
birationalizations of the definitions of Cartier divisor and Weil function.
(Birational variants of Cartier divisors have already been developed as
part of the minimal model program.)
These allow one to finish the proof of Proposition \ref{prop_nevprime_weil},
and therefore to show that $\Nevbir$ can be defined using Weil functions.
(This was the original definition of $\Nevbir$.)

In this section, we define the notion of b-Cartier b-divisor, and show that
the group of these objects (on a fixed variety $X$ over some field),
when partially ordered by the condition that $\mathbf D_1\ge\mathbf D_2$
if $\mathbf D_1-\mathbf D_2$ is effective, forms a lattice (i.e., a partially
ordered set in which every nonempty finite set has a least upper bound and
a greatest lower bound).  We then define the related concept of b-Weil function
on $X$, and show that a corresponding group (obtained by modding out by
the subgroup of $M_k$-bounded functions) is also a lattice, and is naturally
isomorphic to the partially ordered group of b-Cartier b-divisors on $X$.

Once these b-divisors and b-Weil functions have been defined and their
elementary properties discussed, the main result of this section
(Proposition \ref{prop_mu_b_growth_equivs}) is stated and proved.
This result gives alternative descriptions of the main condition of
Definition \ref{def_nevprime} using b-Cartier b-divisors and using
b-Weil functions.

We begin by recalling some definitions from the minimal model program
(the Mori program).  The notion of b-divisor is originally due to Shokurov;
see \cite[Def.~1.7.4 and \S\,2.3]{corti} for details.  The prefix `b'
stands for \emph{birational}.

\begin{definition}\label{def_b_divisor}
Let $X$ be a complete variety over a field $k$.
\begin{enumerate}
\item[(a).]  A \textbf{model} of $X$ is a proper birational morphism $Y\to X$
  over $k$, where $Y$ is a variety over $k$.  We often use $Y$ to denote the
  model.
\item[(b).]  The category of models of $X$ is the category whose objects are
  models of $X$ and whose morphisms are morphisms over $X$.  We say that
  a model $Y_1$ of $X$ \textbf{dominates} a model $Y_2$ of $X$ if there is
  a morphism $Y_1\to Y_2$ (necessarily unique) in this category.
\item[(c).]  A \textbf{b-Cartier b-divisor} on $X$ is an equivalence
  class of pairs $(Y,D)$, where $Y$ is a model of $X$ and $D$ is a
  Cartier divisor on $Y$; here equivalence classes are those for
  the equivalence relation generated by the relation $(Y_1,D_1)\sim (Y_2,D_2)$
  if $Y_1$ dominates $Y_2$ via $\phi\colon Y_1\to Y_2$, and $D_1=\phi^{*}D_2$.
\item[(d).]  A b-Cartier b-divisor $\mathbf D$ on $X$ is \textbf{effective} if
  it is represented by a pair $(Y,D)$ such that $D$ is effective.
\end{enumerate}
\end{definition}

\begin{remark}
Definition~\ref{def_b_divisor}(c) is different from the definition given
in \cite{corti}, but it
is equivalent.  In \emph{op.~cit.,} $X$ is required to be normal, and one
works in the category of normal models of $X$.  A b-divisor on $X$ is an element
$$\mathbf D=(\mathbf D_Y)_{Y} \in \varprojlim_Y \Div(Y)\;,$$
where $\Div(Y)$ is the group of Weil divisors on a normal model $Y$ of $X$,
and the projective limit is relative to push-forwards
$\phi_{*}\colon\Div Y_1\to\Div Y_2$ via morphisms $\phi\colon Y_1\to Y_2$,
if $Y_1$ dominates $Y_2$ via $\phi$.  A b-divisor $\mathbf D$ on $X$ is
b-Cartier if there is a normal model $Y$ of $X$ and a Cartier divisor $D$ on $Y$
such that $\mathbf D_{Y_1}=\phi^{*}D$ for all normal models $Y_1$ of $Y$
with $X$-morphisms $\phi\colon Y_1\to Y$.

To see that this definition is equivalent to Definition~\ref{def_b_divisor}c,
we first note that restricting models in Definition~\ref{def_b_divisor}
to normal models does not change the definition.
Then, since the normalization of $X$ is a final object in the category of
normal models of $X$, we may assume that $X$ is normal.  It is then
straightforward to see that this definition agrees with the definition in
\emph{op.~cit.}
\end{remark}

\begin{lemma}\label{lemm_eff_b_cartier}
Let $X$ be a variety, let $\mathbf D$ be a b-Cartier b-divisor on $X$,
and let $(Y,D)$ be a pair that represents $\mathbf D$.  Assume that $Y$
is normal.  Then $\mathbf D$ is effective if and only if $D$ is effective.

\end{lemma}

\begin{proof}
The reverse implication is immediate from the definitions.

To prove the forward implication, assume that $\mathbf D$ is effective.
By definition, there is a pair $(Y',D')$ representing $\mathbf D$
such that $D'$ is effective.  By pulling back to a possibly larger model,
we may assume that $Y'$ dominates $Y$, say by $f\colon Y'\to Y$,
and that $Y'$ is normal.  Then $D=f_{*}D'$ (here $f_{*}$ refers to Weil
divisors; note that $f_{*}f^{*}D=D$).  In particular, by Remark \ref{remk_eff},
$D$ is effective.
\end{proof}

The following definition generalizes the definition of Weil function to
b-Cartier b-divisors.  This definition comes from \cite[\S\,7]{vojta_ssav1}.
In \emph{loc.~cit.} they were called \emph{generalized Weil functions},
but it is now apparent that it is more natural to call them
\emph{b-Weil functions}.

\begin{definition}\label{def_b_weil_fcn}
Let $X$ be a complete variety over a number field $k$.
Then a \textbf{b-Weil function} on $X$
is an equivalence class of pairs $(U,\lambda)$, where $U$ is a nonempty
Zariski-open subset of $X$ and $\lambda\colon U(M_k)\to\mathbb R$ is
a function such that there exist a model $\phi\colon Y\to X$ of $X$
and a Cartier divisor $D$ on $Y$ such that $\lambda\circ\phi$ extends to
a Weil function for $D$.
Pairs $(U,\lambda)$ and $(U',\lambda')$ are \textbf{equivalent}
if $\lambda=\lambda'$ on $(U\cap U')(M_k)$.
Local b-Weil functions on $X$ are defined similarly.
\end{definition}

It is clear that every Weil function on a variety $X$ over $k$
is also a b-Weil function on $X$, and that b-Weil functions
on $X$ form an abelian group under addition.
Also, if $\phi\colon X\dashrightarrow Y$ is a dominant rational map
and if $\lambda$ is a b-Weil function on $Y$,
then $\phi^{*}\lambda$ (defined in the obvious way) is a
b-Weil function on $X$.  The same facts are true for local b-Weil functions
at a given place $v$.

\begin{definition}\label{def_div_of_b_Weil}
Let $X$ be a complete variety over a number field $k$, let $\lambda$ be
a b-Weil function on $X$, and let $\mathbf D$
be a b-Cartier b-divisor on $X$.  We say that $\lambda$ is a
\textbf{b-Weil function for $\mathbf D$} if $\mathbf D$ is represented by
a pair $(Y,D)$
as in Definition \ref{def_b_divisor}, such that if $\phi\colon Y\to X$
is the structural morphism of $Y$, then $\lambda\circ\phi$ extends to
a Weil function for $D$ on $Y$.
\end{definition}

\begin{proposition}\label{prop_b_div_vs_b_Weil}
Let $X$ be a complete variety over a number field $k$.
\begin{enumerate}
\item[(a).]  For $i=1,2$ let $\mathbf D_i$ be a b-Cartier b-divisor on $X$
and let $\lambda_i$ be a b-Weil function for $\mathbf D_i$.
Then $-\lambda_1$ and $\lambda_1+\lambda_2$ are b-Weil functions
for $-\mathbf D_1$ and $\mathbf D_1+\mathbf D_2$, respectively.
\item[(b).]  Let $\mathbf D$ be a b-Cartier b-divisor and let $\lambda$ be
a b-Weil function on $X$ for $\mathbf D$.  Then $\lambda$ is
$M_k$-bounded from below if and only if $\mathbf D$ is effective,
and $\lambda$ is $M_k$-bounded if and only if $\mathbf D=0$.
\item[(c).]  Let $\lambda$ be a b-Weil function on $X$.
Then there is a unique b-Cartier b-divisor $\mathbf D$ such that $\lambda$
is a b-Weil function for $\mathbf D$.
\item[(d).]  Let $\mathbf D$ be a b-Cartier b-divisor on $X$.  Then there is a
b-Weil function $\lambda$ for $\mathbf D$.
\item[(e).]  The map $\lambda\mapsto\mathbf D$ in part (c) gives a
group isomorphism from the group of b-Weil functions on $X$,
modulo addition of $M_k$-bounded functions, to the group of b-Cartier b-divisors
on $X$.
\end{enumerate}
Analogous statements hold for local b-Weil functions at a fixed place $v$
of $k$.
\end{proposition}

\begin{proof}
(a).  For each $i$ let $\phi_i\colon Y_i\to X$ be a model of $X$ and
let $D_i$ be a Cartier divisor on $Y_i$ such that $\phi_i\circ\lambda_i$
extends to a Weil function on $Y_i$ for $D_i$, and such that
$(Y_i,D_i)$ represents $\mathbf D_i$.  Then the first assertion is immediate,
since $-\phi_1\circ\lambda_1$ extends to a Weil function for $-D_1$
on $Y_1$.  For the second assertion, we may replace $Y_1$ and $Y_2$
with a model $Y$ for $X$ that dominates both of them, and pull back $D_1$
and $D_2$ to $Y$.  Then the assertion follows from additivity of Weil
functions on $Y$.

(b).  Let $\phi\colon Y\to X$ be a model of $X$ and let $D$ be a Cartier
divisor on $Y$ such that $\phi\circ\lambda$ extends to a Weil function
for $D$ on $Y$, and such that $(Y,D)$ represents $\mathbf D$.
By replacing $Y$ with its normalization, we may assume that $Y$ is normal.
Then $\lambda$ is $M_k$-bounded from below if and only if $D$ is effective,
by Proposition \ref{weil_eff}.  If $D$ is effective, then so is $\mathbf D$.
Conversely,
if $\mathbf D$ is effective, then it is represented by a pair $(Y',D')$
with $D'$ effective.  Let $Y''$ be a normal model of $X$ that dominates
both $Y$ and $Y'$, and let $\psi\colon Y''\to Y$ and $\psi'\colon Y''\to Y'$
be the implied morphisms.  Then $D=\psi_{*}(\psi')^{*}D'$ is effective,
so $\lambda$ is $M_k$-bounded from below.

The second assertion follows formally by applying the first assertion
also to $-\lambda$ and $-\mathbf D$.

(c).  Let $\lambda$ be a b-Weil function on $X$.
By Definition~\ref{def_b_weil_fcn} there exist a model $\phi\colon Y\to X$
for $X$ and a Cartier divisor $D$ on $Y$ such that $\lambda\circ\phi$
extends to a Weil function for $D$ on $Y$.  Then $\lambda$ is a b-Weil function
for the b-Cartier b-divisor $\mathbf D$ represented by the pair $(Y,D)$.
To show uniqueness, suppose that $\lambda$ is a b-Weil function for
b-Cartier b-divisors $\mathbf D_1$ and $\mathbf D_2$.  Then $\lambda-\lambda=0$
is a b-Weil function for $\mathbf D_1-\mathbf D_2$.  This divisor must be zero,
by part (b) applied to $\pm(\mathbf D_1-\mathbf D_2)$.

(d).  Let $\mathbf D$ be a b-Cartier b-divisor, and let $(Y,D)$ be a pair
representing it.  By Chow's lemma, we may assume that $Y$ is projective.
By Lang \cite[Ch.~10, Thm.~3.5]{lang}, there is a
Weil function for $D$ on $Y$.  This defines a b-Weil function
for $\mathbf D$.

(e).  Part (c) determines a well-defined function from the group of
b-Weil functions on $X$ to the group of b-Cartier b-divisors on $X$.
Part (a) implies that it is a group homomorphism, part (b)
implies that its kernel is the subgroup of $M_k$-bounded b-Weil functions,
and part (d) implies that it is surjective.

The proofs of corresponding statements for local b-Weil functions are
left to the reader.
\end{proof}

The main reason for defining b-Weil functions in \cite{vojta_ssav1} was
the fact that the (pointwise) maximum of two Weil functions may not be a
Weil function,
but the maximum of two b-Weil functions is another b-Weil function.
The main result of this section shows that the group of b-Cartier b-divisors
on a variety also has a least upper bound, relative to the cone of effective
b-Cartier b-divisors, and that this lub corresponds to the maximum
of b-Weil functions.

We start with a lemma on the underlying geometry of a least upper bound
of b-Cartier b-divisors.  It will not be used until later
(Proposition \ref{prop_mu_b_growth_equivs}), but it is stated here
because it provides valuable intuition.

\begin{lemma}\label{lemma_lub_intuit}
Let $X$ be a variety over a field, and let $\mathbf D$ and
$\mathbf D_1,\dots,\mathbf D_n$ be b-Cartier b-divisors on $X$.
Let $\phi\colon Y\to X$ be a normal model of $X$ such that $\mathbf D$
and $\mathbf D_1,\dots,\mathbf D_\ell$ are represented by Cartier divisors
$D$ and $D_1,\dots,D_\ell$ on $Y$, respectively.  Then $\mathbf D$ is a
least upper bound of $\mathbf D_1,\dots,\mathbf D_\ell$ if and only if
$D-D_i$ is effective for all $i$ and
$$\bigcap_{i=1}^\ell \Supp(D-D_i) = \emptyset\;.$$
\end{lemma}

\begin{proof}
By Lemma \ref{lemm_eff_b_cartier}, $\mathbf D$ is an upper bound
of $\mathbf D_1,\dots,\mathbf D_\ell$ if and only if $D-D_i$ is effective
for all $i$.

Let $Z=\bigcap\Supp(D-D_i)$ and suppose that $Z\ne\emptyset$.
Let $f\colon Y'\to Y$ be the blowing-up of $Y$ along $Z$, and let $E$ be
the exceptional divisor.  Then $E$ is a nonzero effective Cartier divisor,
and $f^{*}(D-D_i)-E$ is effective for all $i$.  This shows that $(Y',f^{*}D-E)$
represents another upper bound of $\mathbf D_1,\dots,\mathbf D_\ell$,
and therefore $\mathbf D$ is not a \emph{least} upper bound.
Thus $Z=\emptyset$.
\end{proof}

This next lemma shows that the above situation is not uncommon.
It is taken from the proof of \cite[Prop.~7.3]{vojta_ssav1}.

\begin{lemma}\label{lemma_max_divisors}
Let $D$ be a Cartier divisor on a variety $Y$ over a field $k$.  Then there
are a proper model $\phi\colon Z\to Y$ and effective Cartier divisors $D'$
and $D''$ on $Z$ such that $\phi^{*}D=D'-D''$, and such that
the supports of $D'$ and $D''$ are disjoint.
\end{lemma}

\begin{proof}  If $U$ is an open subset of $Y$ on which $D$ is equal to
a principal divisor $(f)$, then define $Z_U$ to be the closure of the graph of
the rational function $U\dashrightarrow\mathbb P^1$ given by $f$,
and let $D'$ and $D''$ be the pull-backs of the divisors $[0]$
and $[\infty]$ on $\mathbb P^1$, respectively.  This construction is
compatible with restricting to an open subset $U$, and multiplying
$f$ by an element of $\mathscr O_U^{*}$ induces an automorphism of $Z_U$
that fixes $D'$ and $D''$.  Therefore the schemes $Z_U$ and divisors $D'$
and $D''$ on $Z_U$ glue together to give a scheme $Z$, proper over $Y$,
and Cartier divisors $D'$ and $D''$ on $Z$, that satisfy the conditions
of the lemma.
\end{proof}

For the next step, we recall that a \textbf{lattice} is a partially ordered
set in which every pair of elements $a,b$ has a least upper bound and a
greatest lower bound.  These are called the \textbf{join} and \textbf{meet},
respectively, and are denoted $a\vee b$ and $a\wedge b$, respectively.

The following definition comes from Steinberg \cite[Ch.~2]{steinberg}.

\begin{definition}\label{lattice_group}
A \textbf{lattice-ordered group} is a group $G$, together with a partial
ordering on $G$ that respects the group operation
(i.e., $x\le y\iff xz\le yz\iff zx\le zy$ for all $x,y,z\in G$),
such that the partial ordering forms a lattice.
\end{definition}

In this paper, all lattice-ordered groups are abelian, and are written
additively.

\begin{proposition}\label{prop_b_div_lattice}
Let $X$ be a complete variety over a field $k$.
\begin{enumerate}
\item[(a).]  Let the set of b-Cartier b-divisors on $X$ be partially ordered
by the relation $\mathbf D_1\leq\mathbf D_2$ if $D_2-D_1$ is effective.
Then the group of b-Cartier b-divisors on $X$ is a lattice-ordered group.
\item[(b).]  Assume that $k$ is a number field.
Let $G$ be the group of b-Weil functions on $X$, modulo the
set of $M_k$-bounded functions.  Let $G$ be partially ordered by the condition
that $\lambda_1\leq\lambda_2$ if $\lambda_2-\lambda_1$ is $M_k$-bounded
from below.
Then $G$ is isomorphic to the partially ordered group of b-Cartier b-divisors
on $X$ under the isomorphism of Proposition~\ref{prop_b_div_vs_b_Weil}.
In particular, it is a lattice-ordered group.
\item[(c).]  Assume that $k$ is a number field, and let $G$ be the group of
part (b).  Let $\lambda_1$ and $\lambda_2$ be b-Weil functions for
b-Cartier b-divisors $\mathbf D_1$ and $\mathbf D_2$, respectively, on $X$.
Then the function $\max\{\lambda_1,\lambda_2\}$ is a b-Weil function for
the b-Cartier b-divisor $\mathbf D_1\vee\mathbf D_2$, and its image
in $G$ is the join of the images of $\lambda_1$ and $\lambda_2$ in $G$.
\end{enumerate}
\end{proposition}

\begin{proof}
(a).  That the group is a partially ordered
group is clear from the definition of the ordering.

To check that it is a lattice, by the group property it suffices to check that
for any b-Cartier b-Weil divisor $\mathbf D$ on $X$, the pair
$\mathbf D,\mathbf 0$
has a least upper bound.  To do this, let $(Y,D)$ be a representative
for $\mathbf D$.  By replacing $Y$ with the model constructed in
Lemma~\ref{lemma_max_divisors}, we may assume that $D=D'-D''$, where
$D'$ and $D''$ are effective Cartier divisors with disjoint supports.
We may also assume that $Y$ is normal.  Then $(Y,D')$ represents a
least upper bound for the pair $\mathbf D,\mathbf 0$.  Indeed,
this is true by Lemma \ref{lemma_lub_intuit}, because the divisors $D'-D=D''$
and $D'-0$ are effective divisors with disjoint supports.

(b).  Part (b) of Proposition \ref{prop_b_div_vs_b_Weil} implies that the
group isomorphism preserves the ordering, so $G$ is a lattice-ordered group.

(c).  Again, we may assume that $\lambda_2=0$.  Then $\mathbf D_2=\mathbf 0$.
As in the proof of (a), we may let $(Y,D)$ be a representative
for $\mathbf D_1$, and may assume that $Y$ is normal and that $D=D'-D''$,
where $D'$ and $D''$ are effective with disjoint supports.
Then $(Y,D')$ represents $\mathbf D_1\vee\mathbf 0$.

Let $\phi\colon Y\to X$ be the structural morphism of $Y$.
Then $\phi^{*}\lambda_1$ is a Weil function for $D$.
By Proposition \ref{weil_min}, $\max\{\phi^{*}\lambda_1,0\}$ is a
Weil function for $D'$, and therefore $\max\{\lambda_1,0\}$ is a
b-Weil function for the b-Cartier b-divisor represented by
the pair $(Y,D')$.  Thus, it is a b-Weil function
for $\mathbf D_1\vee\mathbf 0$.  This proves the first assertion.
The other assertion then follows from (b).
\end{proof}

We now can give some equivalent formulations of Definition \ref{def_nevprime}
using b-divisors and b-Weil functions.  We start with a definition that
focuses on the part of the definition of $\Nevbir$ that varies.

\begin{definition}\label{def_mu_b_growth}
Let $X$ be a normal complete variety, let $D$ be an effective Cartier divisor
on $X$, let $\mathscr L$ be a line sheaf on $X$, let $V$ be a linear subspace
of $H^0(X,\mathscr L)$ with $\dim V>1$, and let $\mu>0$ be a rational number.
We say that $D$ has \textbf{$\mu$-b-growth with respect to $V$ and $\mathscr L$}
if there is a model $\phi\colon Y\to X$ of $X$ such that for all $Q\in Y$
there is a basis $\mathcal B$ of $V$ such that
\begin{equation}\label{eq_mu_b_growth}
  \phi^{*}(\mathcal B) \ge \mu \phi^{*}D
\end{equation}
in a Zariski-open neighborhood $U$ of $Q$, relative to the cone of
effective $\mathbb Q$-divisors on $U$.  Also,
we say that $D$ has \textbf{$\mu$-b-growth with respect to $V$}
if it satisfies the above condition with $\mathscr L=\mathscr O(D)$.
\end{definition}

Then Definition \ref{def_nevprime} basically says that $\Nevbir(D)$ is the
infimum of $(\dim V)/\mu$ over all triples $(N,V,\mu)$ such that
$ND$ has $\mu$-b-growth with respect to $V$ and $\mathscr L^D$.
(The corresponding condition for Definition \ref{def_nev2} is called
$\mu$-growth; see \cite{ru}.  The proof of Lemma \ref{lemma_comp_nevs}
then amounts to saying that if $D$ has $\mu$-growth with respect to $V$,
then it also has $\mu$-b-growth with respect to $V$.)

The following proposition completes Proposition \ref{prop_nevprime_weil}
(and adds more equivalent conditions).

\begin{proposition}\label{prop_mu_b_growth_equivs}
Let $X$ be a normal complete variety, let $D$ be an effective Cartier divisor
on $X$, let $\mathscr L$ be a line sheaf on $X$, let $V$ be a linear subspace
of $H^0(X,\mathscr L)$ with $\dim V>1$, and let $\mu>0$ be a rational number.
Then the following are equivalent.
\begin{enumerate}
\item $D$ has $\mu$-b-growth with respect to $V$ and $\mathscr L$.
\item Let $n$ be a positive integer such that $n\mu\in\mathbb Z$.
Then there are bases $\mathcal B_1,\dots,\mathcal B_\ell$ of $V$ such that
\begin{equation}\label{eq_mu_b_growth_ii}
  n\bigvee_{i=1}^\ell (\mathcal B_i) \ge n\mu D
\end{equation}
relative to the cone of effective b-Cartier b-divisors.
\item There are bases $\mathcal B_1,\dots,\mathcal B_\ell$ of $V$;
Weil functions $\lambda_{\mathcal B_1},\dots,\lambda_{\mathcal B_\ell}$
for the divisors $(\mathcal B_1),\dots,(\mathcal B_\ell)$, respectively;
a Weil function $\lambda_D$ for $D$; and an $M_k$-constant $c$ such that
\begin{equation}
  \max_{1\le i\le\ell}\lambda_{\mathcal B_i} \ge \mu\lambda_D - c
\end{equation}
(as functions $X(M_k)\to\mathbb R\cup\{+\infty\}$).
\item For each place $v\in M_k$ there are finitely many bases
$\mathcal B_1,\dots,\mathcal B_\ell$ of $V$; local Weil functions
$\lambda_{\mathcal B_1,v},\dots,\lambda_{\mathcal B_\ell,v}$
for the divisors $(\mathcal B_1),\dots,(\mathcal B_\ell)$, respectively, at $v$;
a local Weil function $\lambda_{D,v}$ for $D$ at $v$; and a constant $c$
such that
\begin{equation}\label{eq_mu_b_growth_iv}
  \max_{1\le i\le\ell}\lambda_{\mathcal B_i,v} \ge \mu\lambda_{D,v} - c
\end{equation}
(as functions $X(\overline k_v)\to\mathbb R\cup\{+\infty\}$).
\item The condition of (iv) holds for at least one place $v$.
\end{enumerate}
\end{proposition}

\begin{proof}
Conditions (ii) and (iii) are equivalent by
Proposition \ref{prop_b_div_lattice}.  Conditions (iii)--(v) are equivalent
by Proposition \ref{weil_eff}.
The implication (i)$\implies$(iii) is Proposition \ref{prop_nevprime_weil}.
Finally, (ii)$\implies$(i) follows from Lemma \ref{lemm_ii_impl_i} (below),
with $D_i=(\mathcal B_i)$ for all $i$.
\end{proof}

\begin{lemma}\label{lemm_ii_impl_i}
Let $X$ be a normal complete variety, let $\mu>0$ be a rational number,
and let $n$ be a positive integer such that $n\mu\in\mathbb Z$.
Let $D$ and $D_1,\dots,D_\ell$ be Cartier divisors on $X$.  Assume that
\begin{equation}\label{lemm_eq_ii}
  n\bigvee_{i=1}^\ell D_i \ge n\mu D
\end{equation}
relative to the cone of effective b-Cartier b-divisors.  Then there is a
model $\phi\colon Y\to X$ of $X$ such that for all $Q\in Y$ there is an
index $i$ such that $\phi^{*}D_i\ge\mu\phi^{*}D$ in a Zariski-open neighborhood
$U$ of $Q$, relative to the cone of effective $\mathbb Q$-divisors on $U$.
\end{lemma}

\begin{proof}
Assume that (\ref{lemm_eq_ii}) is true.
Let $\mathbf E=\bigvee D_i$, and let $\phi\colon Y\to X$ be
a normal model of $X$ such that $\mathbf E$ is represented by a
Cartier divisor $E$ on $Y$.  By Lemma \ref{lemma_lub_intuit},
$E-\phi^{*}D_i$ is effective for all $i$, and
$\bigcap\Supp(E-\phi^{*}D_i)=\emptyset$.
Therefore, for any given $Q\in Y$ there is an index $i$ such that
$Q\notin\Supp(E-\phi^{*}D_i)$.  Fix such an $i$, and
let $U_i=Y\setminus\Supp(E-\phi^{*}D_i)$.  Then $Q\in U_i$.
Moreover, by (\ref{lemm_eq_ii}),
$$n\phi^{*}D_i\big|_{U_i} = nE\big|_{U_i} \ge n\mu\phi^{*}D\big|_{U_i}$$
relative to the cone of effective divisors on $U_i$.  Therefore
$\phi^{*}D_i\ge\mu\phi^{*}D$ on $U_i$ relative to the cone
of effective $\mathbb Q$-divisors on $U_i$, as was to be shown.
\end{proof}

\begin{remark}
In light of Proposition \ref{prop_b_nev}, one may regard (\ref{eq3})
or (\ref{eq_mu_b_growth_ii}) as being conditions that are
\emph{local on the Zariski--Riemann space}, as opposed to (\ref{eq2}),
which is local in the Zariski topology.  (The Zariski--Riemann space of
a complete variety $X$ is the inverse limit of all models of $X$
\cite[Ch.~VI, \S\,17]{zar_sam_II}.)
\end{remark}

Proposition \ref{prop_mu_b_growth_equivs} leads to the following Corollary which is the
Definition \ref{bidef1} of   $\Nevbir(\mathscr L,D)$:

\begin{corollary}\label{cor_nevprime_weil}
Let $X$ be a normal complete variety, let $D$ be an effective Cartier divisor
on $X$, and let $\mathscr L$ be a line sheaf on $X$.  Then
$$\Nevbir(\mathscr L,D) = \inf_{N,V,\mu} \frac{\dim V}{\mu}\;,$$
where the infimum passes over all triples $(N,V,\mu)$ such that
$N\in\mathbb Z_{>0}$, $V$ is a linear subspace of
$H^0(X,\mathscr L^N)$
with $\dim V>1$, and $\mu\in\mathbb Q_{>0}$, with the following property.
There are finitely many bases $\mathcal B_1,\dots,\mathcal B_\ell$
of $V$; Weil functions $\lambda_{\mathcal B_1},\dots,\lambda_{\mathcal B_\ell}$
for the divisors $(\mathcal B_1),\dots,(\mathcal B_\ell)$, respectively;
a Weil function $\lambda_D$ for $D$; and an $M_k$-constant $c$ such that
\begin{equation}
  \max_{1\le i\le\ell}\lambda_{\mathcal B_i} \ge \mu N\lambda_D - c
\end{equation}
(as functions $X(M_k)\to\mathbb R\cup\{+\infty\}$).  (Here we use the same
convention as in Definition \ref{def_nevprime} when there are no triples
$(N,V,\mu)$ that satisfy the condition.)
\end{corollary}

Similar corollaries are true for conditions (ii), (iv), and (v) of
Proposition \ref{prop_mu_b_growth_equivs}.

The following proposition can be used in the proof of the General Theorem
in Section \ref{sect_aut}.

\begin{proposition}\label{prop_aut_weil_ineq}
Let $X$ be a variety over a number field or over $\mathbf C$,
let $\mathscr L$ be a line sheaf on $X$, and let $E_1,\dots,E_m$ be
effective Cartier divisors on $X$.  Let $s$ be a nonzero global section
of $\mathscr L$ lying in the (coherent) subsheaf of $\mathscr L$ generated by
$\mathscr L(-E_j)$ ($j=1,\dots,m$).  Let $\lambda_s$ be a Weil function
for the divisor $(s)$, and for each $j$ let $\lambda_{E_j}$ be a Weil function
for $E_j$.  Then there is an $M_k$-constant $c$ such that
\begin{equation}\label{aut_weil_ineq}
  \lambda_s \ge \min_j \lambda_{E_j} + c\;.
\end{equation}
\end{proposition}

\begin{proof}
By Proposition \ref{prop_b_div_lattice}, (\ref{aut_weil_ineq}) is equivalent
to the assertion that the divisor
$$(s) - \bigwedge_{j=1}^m E_j$$
is effective.

Let $\mathbf E=\bigwedge E_j$, and let $\phi\colon Y\to X$ be a model of $X$
on which $\mathbf E$ is represented by a Cartier divisor $E$.
Since $\phi^{*}E_j-E$ is effective for all $j$,
the sheaf $\phi^{*}\mathscr L(-E)$ contains the sheaves
$\phi^{*}(\mathscr L(-E_j))$ for all $j$, and therefore $\phi^{*}s$ is
a global section of $\phi^{*}\mathscr L(-E)$.

This implies that the divisor $\phi^{*}(s)-E$ is effective, as was to be shown.
\end{proof}

\section{The Proof of Theorems \ref{b_thmd} and \ref{b_thmc}
  for \texorpdfstring{$\Nevbir(\mathscr L, D)$}{Nevbir(L,D)}}\label{proofs}
In this section, we prove Theorems \ref{b_thmd} and \ref{b_thmc}, which are the
variations of Theorems A and B with
$\Nev(D)$ replaced by $\Nevbir(\mathscr L, D)$.
We will only prove  Theorem \ref{b_thmc}  (the number field case), since the proof of  Theorem \ref{b_thmd}  is very similar. However, for this theorem we will give both a complete proof, based on Proposition 5.6 and Theorem B with $\Nev(\mathscr L, D)$ (note that Theorem B still holds,
with the same proof, if $\Nev(D)$ is replaced by $\Nev(\mathscr L, D)$),
and a sketch of how to prove the theorem directly, based on the proof of
Theorem B in \cite{ru}.

\begin{proof}[Proof of Theorem \ref{b_thmc}]
Let $k$, $S$, $X$, and $D$ be as in the statement of the theorem,
and let $\epsilon>0$ be given.  By Proposition \ref{prop_b_nev}b,
there is a model $\phi\colon Y\to X$ of $X$ such that
\begin{equation}\label{b_thmc_eq2}
  \Nev(\phi^*\mathscr L, \phi^{*}D) < \Nevbir(\mathscr L, D) + \epsilon\;.
\end{equation}
Let $Z_0\subseteq Y$ be the ramification locus of $\phi$.

By Theorem B with $\Nev(\mathscr L, D)$, there is a proper Zariski-closed subset $Z_1$ of $Y$
such that the inequality
\begin{equation}\label{b_thmc_eq3}
  m_S(y, \phi^{*}D) \le \left(\Nev(\phi^*\mathscr L, \phi^{*}D)+\epsilon\right) h_{\phi^{*}\mathscr L}(y)
\end{equation}
holds for all $y\in Y(k)$ outside of $Z_1$.

By functoriality of proximity functions, (\ref{b_thmc_eq3}),
(\ref{b_thmc_eq2}), and functoriality of heights, we then have
\[
  \begin{split}
    m_S(x,D) &= m_S(\phi^{-1}(x),\phi^{*}D) + O(1) \\
      &\le \left(\Nev(\phi^*\mathscr L, \phi^{*}D)+\epsilon\right) h_{\phi^{*}\mathscr L}(\phi^{-1}(x))
        + O(1) \\
      &\le \left(\Nevbir(\mathscr L, D)+2\epsilon\right)
	h_{\phi^{*}\mathscr L}(\phi^{-1}(x)) + O(1) \\
      &= \left(\Nevbir(\mathscr L, D)+2\epsilon\right) h_{\mathscr L}(x) + O(1)
  \end{split}
\]
for all $x\in X(k)$ outside of $Z:=\phi(Z_0\cup Z_1)$.  (Note that this set is
closed since $\phi$ is proper, and that $\phi$ induces an isomorphism over
$X\setminus Z$ since $\phi$ is unramified over that set.)
\end{proof}

We now indicate how Theorem \ref{b_thmc} can be proved using the methods of
\cite[Sect.~2]{ru}.

First, we note that it will suffice to prove the following proposition.
This is a variant of \cite[Prop.~2.1]{ru}.

\begin{proposition}\label{b_thmc_prop21}
Let $k$ and $S$ be as in the statement of Theorem \ref{b_thmc}.  Let $X$ be a
normal complete variety over $k$. Let $D$ be an effective Cartier divisor
and $\mathscr L$ be a line sheaf on $X$ with $h^0(\mathscr L^N)\ge 1$ for $N$ big enough.
Let $(N,V,\mu)$ be a triple such that $ND$ has $\mu$-b-growth
with respect to $V$ and $\mathscr L^N$.

Then, for each $\epsilon>0$, there is a proper Zariski-closed subset $Z$ of $X$
such that the inequality
$$m_S(x,D) \le \left(\frac{\dim V}{\mu} + \epsilon\right) h_{\mathscr L}(x)$$
holds for all $x\in X(k)$ outside of $Z$.
\end{proposition}

The fact that this proposition implies Theorem \ref{b_thmc} follows using
almost exactly the same argument as appears in the end of Sect.~2 of \cite{ru}
(involving pulling the proximity and height functions back to the
normalization of $X$).  It will not be repeated here.

\begin{proof}[Proof of Proposition \ref{b_thmc_prop21}]
Assume that $ND$ has $\mu$-b-growth with respect to
$V\subseteq H^0(X,\mathscr L^N)$ and $\mathscr L$.

By Corollary \ref{cor_nevprime_weil},
there are bases $\mathcal B_1,\dots,\mathcal B_\ell$ of $V$;
Weil functions $\lambda_{\mathcal B_1},\dots,\lambda_{\mathcal B_\ell}$
for the divisors $(\mathcal B_1),\dots,(\mathcal B_\ell)$, respectively;
a Weil function $\lambda_D$ for $D$; and an $M_k$-constant $c$ such that
\begin{equation}\label{b_thmc_pr21_eq1}
  \max_{1\le i\le\ell}\lambda_{\mathcal B_i} \ge \mu N\lambda_D - c
\end{equation}
(as functions $X(M_k)\to\mathbb R\cup\{+\infty\}$).

Write
$$\bigcup_{i=1}^\ell \mathcal B_i = \{s_1,\dots,s_q\}\;.$$
and for each $j=1,\dots,q$ choose a Weil function $\lambda_{s_j}$
for the divisor $(s_j)$.

For each $i=1,\dots,\ell$, let $J_i\subseteq\{1,\dots,q\}$ be the subset
such that $\mathcal B_i = \{s_j:j\in J_i\}$.  Then, by (\ref{b_thmc_pr21_eq1}),
for each $v\in S$ there are constants $c_v$ and $c_v'$ such that
\begin{equation}\label{b_thmc_pr21_eq3}
  \begin{split}
  \mu N\lambda_{D,v}
    &\le \max_{1\le i\le\ell} \lambda_{\mathcal B_i,v} + c_v' \\
    &\le \max_{1\le i\le\ell} \sum_{j\in J_i} \lambda_{s_j,v} + c_v\;.
  \end{split}
\end{equation}

By Schmidt's Subspace Theorem in the form of Theorem \ref{schmidt_base},
there is a proper Zariski-closed subset $Z$ of $X$ such that the inequality
\begin{equation}\label{b_thmc_pr21_eq4}
  \sum_{v\in S} \max_J \sum_{j\in J} \lambda_{s_j,v}(x)
    \le (\dim V+\epsilon)h_{\mathscr L^N}(x)
\end{equation}
holds for all $x\in X(k)$ outside of $Z$;
here the maximum is taken over all subsets $J$ of $\{1,\dots,q\}$
for which the sections $s_j$, $j\in J$, are linearly independent.

Combining (\ref{b_thmc_pr21_eq3}) and (\ref{b_thmc_pr21_eq4}) gives
\[
  \begin{split}
    \mu Nm_S(D,x) &= \mu N\sum_{v\in S} \lambda_{D,v}(x) + O(1) \\
      &\le \sum_{v\in S} \max_{1\le i\le\ell} \sum_{j\in J_i} \lambda_{s_j,v}(x)
	+ O(1) \\
      &\le \sum_{v\in S} \max_J \sum_{j\in J} \lambda_{s_j,v}(x) + O(1) \\
      &\le (\dim V+\epsilon)h_{\mathscr L^N}(x) + O(1)\\
      &=N(\dim V+\epsilon)h_{\mathscr L}(x) + O(1)
  \end{split}
\]
for all $x\in X(k)$ outside of $Z$.  Here we used the fact that all of
the $J_i$ occur among the $J$ in (\ref{b_thmc_pr21_eq4}).
\end{proof}

\section{An Example of Faltings: Geometry}\label{falt_geom}

In his 1999 contribution to the \emph{Baker's Garden} volume, Faltings
\cite{faltings} gave a class of
examples of irreducible divisors $D$ on $\mathbb P^2$ for which
$\mathbb P^2\setminus D$ has only finitely many integral points over any
number ring, and over any localization of such a ring away from finitely many
places.

This paper is notable for two reasons.  First, the divisor $D$ is irreducible.
Prior to the paper, the only divisors $D$ on $\mathbb P^2$ for which such
statements were know were divisors with at least four irreducible components.
The second reason is that the paper gives examples of varieties for which
finiteness of integral points is known, yet which cannot be embedded into
semiabelian varieties.  Prior to the paper, the only varieties for which such
finiteness statements were known, and which could not be embedded into
semiabelian varieties, were moduli spaces of abelian varieties.

These examples were further explored by Zannier \cite{zannier} using methods
of Zannier and Corvaja, although Zannier used a different family of examples.
This family has substantial overlap with the examples of Faltings but
does not contain all of his examples.
After that, Levin \cite[\S\,13]{levin_annals}
derived a generalization, using his method of \emph{large divisors,}
that encompasses the examples of both Faltings and Zannier.

Looking at Faltings' examples from the point of view of the Nevanlinna constant
was what led to the formulation of $\Nevbir$.  This came about because
the left-hand sides of (\ref{eq_2.5.1}) and (\ref{eq_2.5.2}) involved
maxima of Weil functions.


We will split the main result of Faltings' paper into two theorems.
The first of the two guarantees that examples with certain properties exist,
and the second says that in each such example, the divisor $D$ on $\mathbb P^2$
has the property that $\mathbb P^2\setminus D$ has only finitely many
integral points over any ring $\mathscr O_{k,S}$, where $\mathscr O_{k,S}$
is the localization $\mathscr O_{k,S}$ of the ring of integers of a
sufficiently large number field $k$ away from a finite set of places $S$.
We prove only the second theorem here, since that is the part that
involves Nevanlinna constants.

This section will cover the geometry of the examples, and the part of the
proof of Theorem \ref{falt_dioph} that is specific to the geometry,
and the next section will finish the proof of the theorem, using methods
applicable in more general settings.

\begin{theorem}[Faltings]\label{falt_constr}
Let $k$ be a field of characteristic zero, and let $X$ be a smooth
geometrically irreducible algebraic surface over $k$.
Then, for all sufficiently positive line sheaves $\mathscr L$ on $X$,
there exists a morphism $f\colon X\to\mathbb P^2$ that satisfies
the following conditions.
\begin{enumerate}
\item  $f^{*}\mathscr O(1)\cong\mathscr L$.
\item  The ramification locus $Z$ of $f$ is smooth and irreducible,
and the ramification index is $2$.
\item  The restriction of $f$ to $Z$ is birational onto its
image $D\subseteq\mathbb P^2$.
\item  $D$ is nonsingular except for cusps and simple double points.
\item  Let $Y\to X\to\mathbb P^2$ denote the Galois closure of
$X\to\mathbb P^2$ (i.e., the normalization of $X$ in the Galois closure
of $K(X)$ over $K(\mathbb P^2)$).  Also let $n=\deg f$.  Then $Y$ is
smooth and its Galois group over $\mathbb P^2$ is the full symmetric group
$\mathscr S_n$.
\item  The ramification locus of $Y$ over $\mathbb P^2$ is the sum of
distinct conjugate effective divisors $Z_{ij}$, $1\le i<j\le n$.  They have
smooth supports, and are disjoint with the following two exceptions.
Points of $Y$ lying over double points of $D$ are fixed points of a subgroup
$\mathscr S_2\times\mathscr S_2$ of $\mathscr S_n$, and they lie
on $Z_{ij}\cap Z_{\ell m}$ with distinct indices $i,j,\ell,m$.
Points of $Y$ lying over cusps of $D$ are fixed points of a subgroup
$\mathscr S_3$ of $\mathscr S_n$, and lie on
$Z_{ij}\cap Z_{i\ell}\cap Z_{j\ell}$.
\end{enumerate}
\end{theorem}

For a proof of this theorem, and also an explicit description of the
``sufficiently positive'' condition on $\mathscr L$,
see Faltings' paper \cite{faltings}.

For convenience, write $Z_{ij}=Z_{ji}$ when $i,j\in\{1,\dots,n\}$ and $i>j$.
Let
$$A_i = \sum_{j\ne i} Z_{ij}
  \qquad\text{and}\qquad M = \sum A_i = \sum_{i\ne j} Z_{ij}\;.$$
Let $L$ be the divisor class of $\mathscr L$ on $X$, and let it also denote
the pull-back of this divisor class to $Y$.  In addition, let $d=\deg D$.
We then have
$$2\sum_{i<j}Z_{ij} = \sum A_i = M \sim dL\;.$$

The other part of Faltings' result is the following.

\begin{theorem}\label{falt_dioph}
Let $k$ be a number field and let $S$ be a finite set of places of $k$.
Let $Y$, $n$, $\{Z_{ij}\}_{i<j}$, $\{A_i\}_i$, and $M$ be as in
Theorem \ref{falt_constr} and the discussion following it.
Also let $\alpha$ be a rational number such that $M-\alpha A_i$ is an ample
$\mathbb Q$-divisor for all $i$.  Then:
\begin{enumerate}
\item[(a).]  if $\alpha>6$ then no set of $\mathscr O_{k,S}$-integral points
on $Y\setminus\bigcup Z_{ij}$ is Zariski-dense, and
\item[(b).]  if $\alpha>8$ then every set of $\mathscr O_{k,S}$-integral points
on $Y\setminus\bigcup Z_{ij}$ is finite.
\end{enumerate}
Since $Y\setminus\bigcup Z_{ij}$ is an \'etale cover of
$\mathbb P^2\setminus D$, the above conclusions also hold for
$\mathbb P^2\setminus D$ (see Serre \cite[\S\,4.2]{serre} or
the second author \cite[\S\,5.1]{Vojta_LNM}).
\end{theorem}

The first part of the proof of this theorem is the following proposition,
which contains all of the geometry specific to the situation of
Theorem \ref{falt_constr}.

\begin{proposition}\label{falt_main_prop}
Let $k$ be a number field, and let $Y$, $n$, $\{Z_{ij}\}_{i<j}$,
$\{A_i\}_i$, $M$, and $\alpha$ be as in Theorem \ref{falt_dioph}.
Assume that $n\ge4$.  Fix Weil functions $\lambda_{ij}$ for each $Z_{ij}$.
Let $\beta$ be an integer such that $\beta\alpha\in\mathbb Z$ and such that
$\beta M$ and all $\beta(M-\alpha A_i)$ are very ample.
Fix an embedding $Y\hookrightarrow\mathbb P^N_k$ associated
to a complete linear system of $\beta M$, and regard $Y$ as a subvariety
of $\mathbb P^N_k$ via this embedding.  Then
\begin{enumerate}
\item[(a).]  There are a finite list $H_1,\dots,H_q$ of hyperplanes in
$\mathbb P^N_k$, with associated Weil functions $\lambda_{H_j}$ for all $j$,
and constants $c_v$ for all $v\in M_k$, with the following property.
Let $\mathscr J$ be the collection of
all $3$-element subsets $J=\{j_0,j_1,j_2\}$ of $\{1,\dots,q\}$ for which
$Y\cap H_{j_0}\cap H_{j_1}\cap H_{j_2}=\emptyset$.
Then $\mathscr J\ne\emptyset$, and the inequality
\begin{equation}\label{eq_2.5.1}
  \max_{J\in\mathscr J}\sum_{j\in J} \lambda_{H_j}(y)
    \ge \beta\alpha \sum_{i<j}\lambda_{ij}(y) - c_v
\end{equation}
holds for all $v\in M_k$ and all $y\in Y(\overline k_v)$ not lying on the
support of any $Z_{ij}$ or on any of the $H_j$.
\item[(b).]  Let $C$ be an integral curve in $Y$, not contained in the support
of any $Z_{ij}$.  Then there are a finite list $H_1,\dots,H_q$ of hyperplanes,
with associated Weil functions as before, and constants $c_v$
for all $v\in M_k$, with the following property.
Let $\mathscr J$ be the collection of all $2$-element subsets $J=\{j_0,j_1\}$
of $\{1,\dots,q\}$ for which $C\cap H_{j_0}\cap H_{j_1}=\emptyset$.
Then $\mathscr J\ne\emptyset$, and the inequality
\begin{equation}\label{eq_2.5.2}
  \max_{J\in\mathscr J}\sum_{j\in J} \lambda_{H_j}(y)
    \ge \frac{\beta\alpha}{2} \sum_{i<j}\lambda_{ij}(y) - c_v
\end{equation}
holds for all $v\in M_k$ and for all but finitely many $y\in C(\overline k_v)$.
\end{enumerate}
\end{proposition}

The proof of this proposition, in turn, relies mainly on two lemmas.
These lemmas replace Faltings' computations of ideals associated to indices.

\begin{lemma}\label{lemma_2.5.3}
Let $i,j,\ell,m$ be distinct indices.  Then:
\begin{enumerate}
\item[(a).]  there exist hyperplanes $H_0$, $H_1$, and $H_2$ in $\mathbb P^N_k$,
such that
$$Y\cap H_0\cap H_1\cap H_2=\emptyset$$
and
\begin{equation}\label{eq_2.5.3.1}
  (H_0+H_1+H_2)\big|_Y - \beta\alpha(Z_{ij}+Z_{\ell m})
\end{equation}
is an effective Cartier divisor on $Y$; and
\item[(b).]  given any integral curve $C\subseteq Y$ not contained in
any of the $Z_{ab}$, there are hyperplanes $H_0$ and $H_1$ in $\mathbb P^N_k$,
such that $C\cap H_0\cap H_1=\emptyset$ and
\begin{equation}\label{eq_2.5.3.2}
  (H_0+H_1)\big|_C - Z_{ij}\big|_C
\end{equation}
is an effective Cartier divisor on $C$.
\end{enumerate}
\end{lemma}

\begin{proof}
Let $\sigma_i$ and $\sigma_j$ be the canonical sections of $\mathscr O(A_i)$
and $\mathscr O(A_j)$, respectively.  Then the linear system
$$\sigma_i^{\beta\alpha}\cdot\Gamma(Y,\beta(M-\alpha A_i))
  + \sigma_j^{\beta\alpha}\cdot\Gamma(Y,\beta(M-\alpha A_j))$$
has base locus $\Supp A_i\cap\Supp A_j$, since the first summand has
base locus $\Supp A_i$ and the second has base locus $\Supp A_j$.
This intersection consists of the union of $Z_{ij}$ and finitely many closed
points.  Choose an element of this linear system, sufficiently generic so that
it does not vanish identically on any irreducible component of $Z_{\ell m}$,
and let $H_0$ be the associated hyperplane in $\mathbb P^N_k$.  Then
$H_0\big|_Y - \beta\alpha Z_{ij}$ is an effective divisor.

Similarly let $\sigma_\ell$ and $\sigma_m$ be the canonical sections of
$\mathscr O(A_\ell)$ and $\mathscr O(A_m)$, and let $H_1$ be the
hyperplane associated to an element of
$$\sigma_\ell^{\beta\alpha}\cdot\Gamma(Y,\beta(M-\alpha A_\ell))
  + \sigma_m^{\beta\alpha}\cdot\Gamma(Y,\beta(M-\alpha A_m))\;,$$
chosen sufficiently generically such that $H_1$ does not contain any
irreducible component of $H_0\cap Y$.
Then $H_1\big|_Y - \beta\alpha Z_{\ell m}$ is effective.

By construction, $Y\cap H_0\cap H_1$ is a finite union of closed points,
so we can let $H_2$ be a hyperplane that avoids those points to ensure that
$Y\cap H_0\cap H_1\cap H_2=\emptyset$.  By construction,
$$\left(H_0\big|_Y - \beta\alpha Z_{ij}\right)
  + \left(H_1\big|_Y - \beta\alpha Z_{\ell m}\right) + H_2\big|_Y$$
is effective, and this is the divisor (\ref{eq_2.5.3.1}).  This proves (a).

For part (b), let $\sigma_i$ be as above, and let $H_0$ be the hyperplane
associated to an element of
$\sigma_i^{\beta\alpha}\cdot\Gamma(Y,\beta(M-\alpha A_i))$,
chosen generically so that $H_0$ does not contain $C$.  Let $H_1$ be a
hyperplane in $\mathbb P^N_k$, chosen so that $C\cap H_0\cap H_1=\emptyset$.
Since $H_0\big|_C - Z_{ij}\big|_C$ is an effective divisor,
so is (\ref{eq_2.5.3.2}).
\end{proof}

\begin{lemma}\label{lemma_2.5.4}
Let $i,j,\ell$ be distinct indices.  Then:
\begin{enumerate}
\item[(a).]  there exist hyperplanes $H_0$, $H_1$, and $H_2$ in $\mathbb P^N_k$,
such that
$$Y\cap H_0\cap H_1\cap H_2=\emptyset$$
and
$$(H_0+H_1+H_2)\big|_Y
  - \beta\alpha(Z_{ij} + Z_{i\ell} + Z_{j\ell})$$
is an effective Cartier divisor on $Y$; and
\item[(b).]  given any integral curve $C\subseteq Y$ not contained in any of
the $Z_{ab}$, there are hyperplanes $H_0$ and $H_1$ in $\mathbb P^N_k$,
such that $C\cap H_0\cap H_1=\emptyset$ and
\begin{equation}\label{eq_2.5.4.1}
  (H_0+H_1)\big|_C - \beta\alpha(Z_{ij} + Z_{i\ell})\big|_C
\end{equation}
is an effective Cartier divisor on $C$.
\end{enumerate}
\end{lemma}

\begin{proof}
Let $\sigma_i$ and $\sigma_j$ be as in the preceding proof.
Choose a section of the linear system
$$\sigma_i^{\beta\alpha}\cdot\Gamma(Y,\beta(M-\alpha A_i))
  + \sigma_j^{\beta\alpha}\cdot\Gamma(Y,\beta(M-\alpha A_j))\;,$$
and let $H_0$ be the associated hyperplane.
Then $H_0\big|_Y-\beta\alpha Z_{ij}$ is effective.
We may assume that the choice of $H_0$ is sufficiently generic
so that $H_0$ does not contain any irreducible component of $A_\ell$.

Next let $\sigma_\ell$ be the canonical section of $\mathscr O(A_\ell)$,
and let $H_1$ be the hyperplane associated to a section of
$$\sigma_\ell^{\beta\alpha}\cdot\Gamma(Y,\beta(M-\alpha A_\ell))\;.$$
Then $H_1\big|_Y - \beta\alpha(Z_{i\ell}+Z_{j\ell})$ is effective.
We may assume that $H_1$ does not contain any irreducible component
of $Y\cap H_0$.

Again, $Y\cap H_0\cap H_1$ consists of finitely many points, and we choose
$H_2$ to be any hyperplane not meeting any of these points.  Part (a)
then concludes as in the previous lemma.

For part (b), let $H_0$ and $H_1$ be the hyperplanes associated to suitably
chosen sections of
$\sigma_i^{\beta\alpha}\cdot\Gamma(Y,\beta(M-\alpha A_i))$ and
$\Gamma(Y,\beta M)$, respectively.  As in the previous lemma, we then have
$C\cap H_0\cap H_1=\emptyset$.  Since
$H_0\big|_C-\beta\alpha(Z_{ij}+Z_{i\ell})\big|_C$ is effective,
so is (\ref{eq_2.5.4.1}).
\end{proof}

\begin{proof}[Proof of Proposition \ref{falt_main_prop}]
First consider part (a) of the proposition.

Fix a place $v\in M_k$.
Apply Lemmas \ref{lemma_2.5.3}a and \ref{lemma_2.5.4}a to all possible
collections $i,j,\ell,m$ and $i,j,\ell$ of indices, respectively.
This involves only finitely many applications, so only finitely many
hyperplanes occur.  Let $H_1,\dots,H_q$ be those hyperplanes.

The conditions in Theorem \ref{falt_constr} on the intersections of
the divisors $Z_{ij}$
imply that there is a constant $C_v$ such that, for each $y\in Y(k)$
not in $\bigcup\Supp Z_{ij}$, one of the following conditions holds.
\begin{enumerate}
\item  $\lambda_{ij}(y)\le C_v$ for all $i$ and $j$;
\item  there are indices $i$ and $j$ such that $\lambda_{ij}(y)>C_v$
but $\lambda_{ab}(y)\le C_v$ in all other cases;
\item  there are distinct indices $i$, $j$, $\ell$, and $m$ such that
$\lambda_{ij}(y)>C_v$ and $\lambda_{\ell m}(y)>C_v$
but $\lambda_{ab}(y)\le C_v$ in all other cases; or
\item  there are indices $i,j,\ell$ such that
$\max\{\lambda_{ij}(y), \lambda_{i\ell}(y), \lambda_{j\ell}(y)\}>C_v$,
but $\lambda_{ab}(y)\le C_v$ if $\{a,b\}\nsubseteq\{i,j,\ell\}$.
\end{enumerate}

For case (iii), (\ref{eq_2.5.1}) follows from Lemma \ref{lemma_2.5.3}a,
since one can take
$J$ corresponding to the hyperplanes occurring in the lemma,
and the inequality will then follow from effectivity of (\ref{eq_2.5.3.1}).
Case (ii) follows as a special case of this lemma, since $n\ge 4$.
Case (iv) follows from Lemma \ref{lemma_2.5.4}a, by a similar argument.
Finally, in case (i) there is nothing to prove.  This proves (a).

For part (b), let $H_1,\dots,H_q$ be a finite collection of hyperplanes
occurring in all possible applications of Lemmas \ref{lemma_2.5.3}b
and \ref{lemma_2.5.4}b
with the given curve $C$.  We have cases (i)--(iv) as before.
Cases (ii) and (iii) follow from Lemma \ref{lemma_2.5.3}b, where we may assume
without loss of generality that $\lambda_{ij}(y)\ge\lambda_{\ell m}(y)$
to obtain (\ref{eq_2.5.2}) from effectivity of (\ref{eq_2.5.3.2}).
Similarly, case (iv)
follows from Lemma \ref{lemma_2.5.4}b after a suitable
permutation of the indices, and case (i) is again trivial.
\end{proof}

\begin{remark}
It will not actually be needed in the sequel, but the collections $(c_v)$ of
constants in each part of Proposition \ref{falt_main_prop} are actually
$M_k$-constants.  This follows directly from Proposition \ref{weil_eff}.
\end{remark}

\begin{remark}
It is possible (and, in fact, slightly easier) to write
Proposition \ref{falt_main_prop} in terms of b-Cartier b-divisors
instead of Weil functions.  For example, one can replace (\ref{eq_2.5.1}) with
\begin{equation}\label{eq_2_5_1bis}
  \bigvee_{J\in\mathscr J} \sum_{j\in J} H_j \ge \beta\alpha\sum_{i<j} Z_{ij}
\end{equation}
relative to the cone of effective b-Cartier b-divisors.

In the proof, one would let $\phi\colon Y\to X$ be a model for which the
left-hand side of (\ref{eq_2_5_1bis}) pulls back to an ordinary divisor,
and for each $P\in X$ one would consider four cases:
\begin{enumerate}
\item $P\notin Z_{ij}$ for all $i$, $j$;
\item $P\in Z_{ij}$ for exactly one pair $i,j$;
\item there are distinct indices $i,j,\ell,m$ such that $P\in Z_{ij}$
  and $P\in Z_{\ell m}$, but $P\notin Z_{ab}$ for all other components; and
\item there are indices $i,j,\ell$ such that $P$ lies on at least two of
  $Z_{ij}$, $Z_{i\ell}$, and $Z_{j\ell}$, but $P\notin Z_{ab}$ if
  $\{a,b\}\nsubseteq\{i,j,\ell\}$.
\end{enumerate}
In each case let $U$ be the complement of all $Z_{ab}$ that do not contain $P$.
Then $U$ is an open neighborhood of $P$ in $X$, and there is a set $J$
of indices such that $\sum_{j\in J}H_j-\beta\alpha\sum_{a<b}Z_{ab}$ is
effective on $U$.  This set $J$ is obtained from Lemmas \ref{lemma_2.5.3}
or \ref{lemma_2.5.4}, as appropriate.

We choose to keep the phrasing in terms of Weil functions, however, since
that is the phrasing that will be most convenient for the next step in the
proof of Theorem \ref{falt_dioph}.
\end{remark}

\section{Conclusion of the Proof of Theorem \ref{falt_dioph}}\label{falt_finis}

Although it is possible to conclude from Proposition \ref{falt_main_prop}
that $\Nevbir\left(\sum Z_{ij}\right)<1$, this does not follow nearly
as directly
as one might first hope.  One needs to use Mumford's theory of degree
of contact, as in the work of Evertse and Ferretti \cite{ef_imrn}.
This leads to a definition of \emph{Evertse--Ferretti Nevanlinna constant}
$\NevEF(\mathscr L,D)$ (Definition \ref{def_nevef}).

This section defines this new Nevanlinna constant, shows that it satisfies
the expected diophantine property (Theorem \ref{ef_thmc}), and uses this theory
to complete the proof of Theorem \ref{falt_dioph}.

We start with a definition that corresponds to that of $\mu$-b-growth,
and is suitable for applying the work of Evertse and Ferretti.

\begin{definition}\label{def_mu_ef_growth}
Let $X$ be a normal complete variety, let $D$ be an effective Cartier divisor
on $X$, let $\mathscr L$ be a line sheaf on $X$, and let $\mu>0$ be
a rational number.  We say that $D$ has
\textbf{$\mu$-EF-growth with respect to $\mathscr L$} if there is
a model $\phi\colon Y\to X$ of $X$ such that for all $Q\in Y$ there are
a base-point-free linear subspace $V\subseteq H^0(X,\mathscr L)$
with $\dim V=\dim X+1$ and a basis $\mathcal B$ of $V$ such that
\begin{equation}\label{eq_mu_ef_growth}
  \phi^{*}(\mathcal B) \ge \mu \phi^{*}D
\end{equation}
in a Zariski-open neighborhood $U$ of $Q$, relative to the cone of
effective $\mathbb Q$-divisors on $U$.
Also, we say that $D$ has \textbf{$\mu$-EF-growth}
if it satisfies the above condition with $\mathscr L=\mathscr O(D)$.
\end{definition}

As with Proposition \ref{prop_mu_b_growth_equivs}, this definition can be
expressed equivalently using b-Cartier b-divisors or b-Weil functions.

\begin{proposition}\label{prop_mu_ef_growth_equivs}
Let $X$ be a normal complete variety, let $D$ be an effective Cartier divisor
on $X$, let $\mathscr L$ be a line sheaf on $X$, and let $\mu>0$ be a
rational number.  Then the following are equivalent.
\begin{enumerate}
\item $D$ has $\mu$-EF-growth with respect to $\mathscr L$.
\item Let $n$ be a positive integer such that $n\mu\in\mathbb Z$.
Then there are base-point-free linear subspaces $V_1,\dots,V_\ell$
of $H^0(X,\mathscr L)$, all of dimension $\dim X+1$, and corresponding bases
$\mathcal B_1,\dots,\mathcal B_\ell$ of $V_1,\dots,V_\ell$, respectively,
such that
\begin{equation}\label{eq_mu_ef_growth_ii}
  n\bigvee_{i=1}^\ell (\mathcal B_i) \ge n\mu D
\end{equation}
relative to the cone of effective b-Cartier b-divisors.
\item There are base-point-free linear subspaces $V_1,\dots,V_\ell$
of $H^0(X,\mathscr L)$, all of dimension $\dim X+1$;
bases $\mathcal B_1,\dots,\mathcal B_\ell$ of $V_1,\dots,V_\ell$, respectively;
Weil functions $\lambda_{\mathcal B_1},\dots,\lambda_{\mathcal B_\ell}$
for the divisors $(\mathcal B_1),\dots,(\mathcal B_\ell)$, respectively;
a Weil function $\lambda_D$ for $D$; and an $M_k$-constant $c$ such that
\begin{equation}\label{eq_mu_ef_growth_iii}
  \max_{1\le i\le\ell}\lambda_{\mathcal B_i} \ge \mu\lambda_D - c
\end{equation}
(as functions $X(M_k)\to\mathbb R\cup\{+\infty\}$).
\item For each $v\in M_k$ there are base-point-free linear subspaces
$V_1,\dots,V_\ell$ of $H^0(X,\mathscr L)$, all of dimension $\dim X+1$;
bases $\mathcal B_1,\dots,\mathcal B_\ell$ of $V_1,\dots,V_\ell$, respectively;
local Weil functions
$\lambda_{\mathcal B_1,v},\dots,\lambda_{\mathcal B_\ell,v}$ at $v$
for the divisors $(\mathcal B_1),\dots,(\mathcal B_\ell)$, respectively;
a local Weil function $\lambda_{D,v}$ for $D$ at $v$; and a constant $c$
such that
\begin{equation}\label{eq_mu_ef_growth_iv}
  \max_{1\le i\le\ell}\lambda_{\mathcal B_i,v} \ge \mu\lambda_{D,v} - c
\end{equation}
(as functions $X(\overline k_v)\to\mathbb R\cup\{+\infty\}$).
\item The condition of (iv) holds for at least one place $v$.
\end{enumerate}
\end{proposition}

\begin{proof}
Conditions (ii) and (iii) are equivalent by
Proposition \ref{prop_b_div_lattice}, and (iii)--(v) are equivalent
by Proposition \ref{weil_eff}.
The implication (ii)$\implies$(i) follows from Lemma \ref{lemm_ii_impl_i},
with $D_i=(\mathcal B_i)$ for all $i$.

To finish the proof, it will suffice to show that (i) implies (ii).

Assume that condition (i) is true.  Let $\phi\colon Y\to X$ be a model
that satisfies the condition of Definition \ref{def_mu_ef_growth}.
By quasi-compactness of $Y$,
we may assume that only finitely many triples $(U,V,\mathcal B)$ occur.
Let $(U_1,V_1,\mathcal B_1),\dots,(U_\ell,V_\ell,\mathcal B_\ell)$
be those triples.  We may assume that $\bigvee(\mathcal B_i)$ is represented
by a Cartier divisor $E$ on $Y$, and that $Y$ is normal.
Let $n$ be a positive integer such that $n\mu\in\mathbb Z$, and let $n'$
be a positive integer such that $nn'\phi^{*}(\mathcal B_i)-nn'\mu\phi^{*}D$
is effective for all $i$.  Then, for each $i$,
$$nn'E\big|_{U_i} \ge nn'\phi^{*}(\mathcal B_i)\big|_{U_i}
  \ge nn'\mu\phi^{*}D\big|_{U_i}$$
relative to the cone of effective Cartier divisors on $U_i$,
and this implies (\ref{eq_mu_ef_growth_ii}) since $Y$ is normal.
\end{proof}

\begin{definition}\label{def_nevef}
Let $X$ be a complete variety, let $D$ be an effective Cartier divisor
on $X$, and let $\mathscr L$ be a line sheaf on $X$.  If $X$ is normal, then
we define
$$\NevEF(\mathscr L,D) = \inf_{N,\mu} \frac{\dim X+1}{\mu}\;,$$
where the infimum passes over all pairs $(N,\mu)$ such that
$N\in\mathbb Z_{>0}$, $\mu\in\mathbb Q_{>0}$, and $ND$ has $\mu$-EF-growth
with respect to $\mathscr L^N$.
If there are no such pairs $(N,\mu)$, then $\NevEF(\mathscr L,D)$ is defined
to be $+\infty$. For a general complete variety $X$, $\NevEF(\mathscr L,D)$
is defined by pulling back to the normalization of $X$.

If $L$ is a Cartier divisor or Cartier divisor class on $X$, then
we define $\NevEF(L,D)=\NevEF(\mathscr O(L),D)$.  We also define
$\NevEF(D)=\NevEF(D,D)$.
\end{definition}

\begin{remark}\label{remk_EF_D_linear}
As in Remark \ref{remk_D_linear}, if $n$ is a positive integer then
$$\NevEF(\mathscr L,nD) = n\NevEF(\mathscr L,D)\;.$$
\end{remark}

The two parts of Proposition \ref{falt_main_prop} say that
$\beta\sum Z_{ij}=(\beta/2)M$ has $\alpha$-EF-growth and $(\alpha/2)$-EF-growth,
respectively, with respect to $\mathscr O(\beta M)$, and therefore
\[
  \NevEF\left(\beta M,\frac\beta2M\right) \le \frac3\alpha
    \qquad\text{and}\qquad
    \NevEF\left(\beta M,\frac\beta2M\right) \le \frac4\alpha\;,
\]
respectively, on $Y$ and $C$, respectively.  By Remark \ref{remk_EF_D_linear},
these become
\begin{equation}\label{eq_falt_nevef}
  \NevEF(\beta M) \le \frac6\alpha \qquad\text{and}\qquad
    \NevEF(\beta M) \le \frac8\alpha\;,
\end{equation}
respectively.

The main result of this section, Theorem \ref{ef_thmc}, is proved by
first showing that $\Nevbir(\mathscr L,D) \le \NevEF(\mathscr L,D)$
(Theorem \ref{thm_ineq_b_nev_ef_nev}).  This, in turn, reduces to
comparing the properties of $\mu$-b-growth and $\mu$-EF-growth.

\begin{proposition}\label{prop_b_vs_ef_growth}
Let $X$ be a normal variety over a number field $k$, let $D$ be an effective
Cartier divisor on $X$, let $\mathscr L$ be a line sheaf on $X$, and let
$\mu>0$ be a rational number.  Assume that $D$ has $\mu$-EF-growth with
respect to $\mathscr L$.  Then for all $\epsilon>0$,
there are a positive integer $m$, a rational number $\nu$, and
a linear subspace $V\subseteq H^0(X,\mathscr L^m)$ such that
$mD$ has $\nu$-b-growth with respect to $V$ and $\mathscr L^m$,
and such that
\begin{equation}\label{eq_b_vs_ef_nu_bound}
  \frac{\dim V}{\nu} \le \frac{\dim X + 1}{\mu} + \epsilon\;.
\end{equation}
\end{proposition}

\begin{proof}
Assume that $D$ has $\mu$-EF-growth with respect to $\mathscr L$.
Let $V_1,\dots,V_\ell$; $\mathcal B_1,\dots,\mathcal B_\ell$;
$\lambda_{\mathcal B_1},\dots,\lambda_{\mathcal B_\ell}$; and $\lambda_D$
be as in condition (iii) of Proposition \ref{prop_mu_ef_growth_equivs}.

Let $\{s_1,\dots,s_q\}$ be the elements of $\bigcup_i\mathcal B_i$.
These sections determine a morphism $\Phi\colon X\to\mathbb P^{q-1}$.
Let $Y$ be the image.  Note that $\Phi$ need not be a closed immersion;
in fact, it may happen that $\dim Y<\dim X$.  Let $n=\dim Y$.

From the theory of degree of contact
(see \cite[Sects.~3--4]{ef_imrn} or \cite[Sect.~3]{ru} for the definitions),
there is a constant $C>0$ such that the inequality
\begin{equation}\label{eq_ef_vs_b_1}
  c_{j_0}+\dots+c_{j_n}
  \le \frac{n+1}{m}\frac{S_Y(m,\mathbf c)}{H_Y(m)}\left(1+\frac Cm\right)
\end{equation}
holds for all integers $m>0$, all $q$-tuples $\mathbf c\in\mathbb R_{\ge0}^q$,
and all $(j_0,\dots,j_n)$ for which the divisors $(s_{j_0}),\dots,(s_{j_n})$
are base-point-free on $X$.

Fix a place $v$ of $k$.

For each $j=1,\dots,q$, choose a local Weil function $\lambda_{(s_j),v}$ for
the divisor $(s_j)$ at $v$.  Since each $(s_j)$ is effective, we may assume
that each $\lambda_{(s_j)}$ is nonnegative.


We will apply (\ref{eq_ef_vs_b_1}) with
$\mathbf c=(\lambda_{(s_1),v}(x),\dots,\lambda_{(s_q),v}(x))$
for some $x\in X(\overline k_v)$.
By (\ref{eq_mu_ef_growth_iv}), there is an index $i$ such that
\begin{equation}\label{eq_ef_vs_b_2}
  \lambda_{\mathcal B_i,v}(x) \ge \mu\lambda_{D,v}(x) + O(1)\;,
\end{equation}
where the implicit constant does not depend on $x$ (or $i$).
Write $\mathcal B_i=\{s_{j_0},\dots,s_{j_n}\}$.  Then
\begin{equation}\label{eq_ef_vs_b_3}
  c_{j_0}+\dots+c_{j_n}
    = \lambda_{(s_{j_0}),v}(x)+\dots+\lambda_{(s_{j_n}),v}(x)
    = \lambda_{\mathcal B_i,v}(x) + O(1)\;,
\end{equation}
where the implicit constant does not depend on $x$.
Since there are only finitely many possible values for $i$, the constant may
also be taken independent of $i$.

Combining (\ref{eq_ef_vs_b_1}), (\ref{eq_ef_vs_b_3}), and (\ref{eq_ef_vs_b_2})
gives
\[
  \begin{split} S_Y(m,\mathbf c)
    &\ge \frac{mH_Y(m)}{(n+1)(1+C/m)}\mu\lambda_{D,v}(x) + O(1) \\
    &= \nu m\lambda_{D,v}(x) + O(1)\;,
  \end{split}
\]
where again the implicit constant does not depend on $x$, and
$$\nu = \frac{\mu H_Y(m)}{(n+1)(1+C/m)}\;.$$
Let $V$ be the pull-back of $H^0(Y,\mathscr O(m))$.  Then $\dim V=H_Y(m)$.
Since $\dim X\ge n$, we have
\[
  \frac{\dim V}{\nu} = \frac{(n+1)(1+C/m)}{\mu}
    \le \frac{\dim X+1}{\mu}\left(1+\frac Cm\right)\;.
\]
Thus (\ref{eq_b_vs_ef_nu_bound}) holds for sufficiently large $m$.

On the other hand, by the definition of $S_Y(m,\mathbf c)$ and
our choice of $\mathbf c$, there are bases $\mathcal B_1,\dots,\mathcal B_r$
of $V$ and corresponding local Weil functions
$\lambda_{\mathcal B_1,v},\dots,\lambda_{\mathcal B_r,v}$ such that
\[
  S_Y(m,\mathbf c) = \max_{1\le i\le r}\lambda_{(\mathcal B_i),v}(\Phi(x))
\]
for all $x\in X(\overline k_v)$.
Thus, after pulling the bases back to $V$ and the local Weil functions back
to $X$, we see that $S_Y(m,\mathbf c)$ equals the left-hand side of
(\ref{eq_mu_b_growth_iv}), and hence $mD$ has $\nu$-b-growth
with respect to $V$ and $\mathscr L^m$.
\end{proof}

This proposition then leads quickly to the main results of the section.

\begin{theorem}\label{thm_ineq_b_nev_ef_nev}
Let $X$ be a variety over a number field $k$, let $D$ be an effective Cartier
divisor on $X$, and let $\mathscr L$ be a line sheaf on $X$.  Then
\begin{equation}\label{ineq_b_nev_ef_nev}
  \Nevbir(\mathscr L,D) \le \NevEF(\mathscr L,D)\;.
\end{equation}
\end{theorem}

\begin{proof}
We may assume that $X$ is normal, since the definitions of $\Nevbir$
and $\NevEF$ both handle the general case by pulling back to the normalization.
We may also assume that $\NevEF(\mathscr L,D)<\infty$
(otherwise there is nothing to prove).

Let $\epsilon>0$.  By definition of $\NevEF$, there is a pair $(N,\mu)$
with $N\in\mathbb Z_{>0}$ and $\mu\in\mathbb Q_{>0}$, such that
$ND$ has $\mu$-EF-growth with respect to $\mathscr L^N$ and such that
$$\frac{\dim X + 1}{\mu} < \NevEF(\mathscr L,D) + \epsilon\;.$$
By Proposition \ref{prop_b_vs_ef_growth}, there are a positive integer $m$,
a rational number $\nu$, and a linear subspace
$V\subseteq H^0(X,\mathscr L^{mN})$ such that $mND$ has $\nu$-b-growth
with respect to $V$ and $\mathscr L^{mN}$, and such that
$$\frac{\dim V}{\nu} \le \frac{\dim X + 1}{\mu} + \epsilon\;.$$
We then have
$$\Nevbir(\mathscr L,D) \le \frac{\dim V}{\nu}
  < \NevEF(\mathscr L,D) + 2\epsilon\;,$$
and the proof concludes by letting $\epsilon$ go to zero.
\end{proof}

\begin{theorem}\label{ef_thmc}
Let $k$ be a number field, and let $S$ be a finite set of places of $k$
containing all archimedean places.
Let $X$ be a projective variety over $k$, and let $D$ be an effective
Cartier divisor on $X$.  Then, for every $\epsilon>0$,
there is a proper Zariski-closed subset $Z$ of $X$ such that the inequality
\begin{equation}
  m_S(x, D) \le \left(\NevEF(D)+\epsilon\right) h_D(x)
\end{equation}
holds for all $x\in X(k)$ outside of $Z$.
\end{theorem}

\begin{proof}
This is immediate from (\ref{ineq_b_nev_ef_nev}) and Theorem \ref{b_thmc}.
\end{proof}

\begin{corollary}\label{cor_nev_ef_int_pts}
Let $X$ be a projective variety, and let $D$ be an ample Cartier divisor
on $X$.  If $\NevEF(D)<1$ then there is a proper Zariski-closed subset $Z$
of $X$ such that any set of $D$-integral points on $X$ has only finitely many
points outside of $Z$.
\end{corollary}

Finally, we are able to prove Theorem \ref{falt_dioph}.

\begin{proof}[Proof of Theorem \ref{falt_dioph}]
(a).  By (\ref{eq_falt_nevef}) and the assumption $\alpha>6$, we have
$$\NevEF(\beta M) \le \frac6\alpha < 1\;.$$
The result then follows by Corollary \ref{cor_nev_ef_int_pts}.

(b).  Let $Z$ be the Zariski closure of a set of $D$-integral points on $Y$.
By part (a), $Z\ne Y$, so it will suffice to show that no irreducible component
of $Z$ can be a curve.  This holds because, on any curve $C$ in $Y$
not contained in $\Supp M$,
$$\NevEF(\beta M) \le \frac8\alpha < 1\;,$$
and we conclude as before.
\end{proof}

\printbibliography

\end{document}